\documentclass[11pt,letterpaper,leqno]{article}
\usepackage{tikz}
\usepackage{amsmath}
\usepackage{layout}
\usepackage{amsthm}
\usepackage{fancyhdr}

\usepackage{amsfonts,csquotes}

\fancyheadoffset{0cm}
\newcommand{\nvd}{Z} 

\newtheorem{theorem}{Theorem}[section]
\newtheorem{lemma}[theorem]{Lemma}
\newtheorem{definition}{Definition} [section]

\begin{document}

\title{Properties of stochastic Kronecker graphs}

\author{Mihyun Kang$^1$\footnote{supported by DFG KA 2748/3-1 and Austrian Science Fund (FWF): W1230, P26826}, Micha\l {\ }Karo\'nski$^2$, Christoph Koch$^1$\footnote{supported by NAWI Graz and Austrian Science Fund (FWF): P26826} , Tam\'as Makai$^1$\footnote{supported by DFG KA 2748/3-1 and Austrian Science Fund (FWF): P26826}\\
\\
$^1$ Graz University of Technology\\
Institute of Optimization and Discrete Mathematics\\
Steyrergasse 30,  8010 Graz, Austria\\
	\{kang, ckoch, makai\}@math.tugraz.at\\
$^2$ Adam Mickiewicz University\\
Department of Discrete Mathematics\\
Umultowska 87, 61-614 Pozna\'n \\
karonski@amu.edu.pl 
}

\maketitle
\begin{abstract}
The stochastic Kronecker graph model introduced by Leskovec et al. is a
random graph with vertex set $\mathbb Z_2^n$, where two vertices $u$ and $v$ are connected
with probability 
$\alpha^{{u}\cdot{v}}\gamma^{(1-{u})\cdot(1-{v})}\beta^{n-{u}\cdot{v}-(1-{u})\cdot(1-{v})}$
independently of the presence or absence of any other edge, for fixed parameters $0<\alpha,\beta,\gamma<1$. They have shown empirically that the degree sequence resembles a power law
degree distribution. In this paper we show that the stochastic Kronecker graph a.a.s.\ does not feature a power law degree distribution for any parameters $0<\alpha,\beta,\gamma<1$. In addition, we analyze the number of subgraphs present in the stochastic Kronecker graph and study the typical neighborhood of any given vertex.
\end{abstract}
MSC Primary: 05C80

\section{Introduction}

Kronecker graphs were introduced by Leskovec, Chakrabarti, Kleinberg and Faloutsos \cite{Kroneckerintr} in order to model real world networks. First they considered a deterministic model based on Kronecker multiplication which creates graphs exhibiting several properties of real world networks like heavy tailed degree distribution and average degree that grows as a power law with the size of the graph. They also introduced the random version of this model, called the stochastic Kronecker graph.

Let $n\in \mathbb N$ and $0<\alpha,\beta,\gamma<1$ be probabilities and define 
\[P=\begin{pmatrix} \alpha & \beta \\ \beta &\gamma \end{pmatrix}.\]
The stochastic Kronecker graph $K(n,P)$  is a graph whose vertex set is given by the set $\mathbb{Z}_2^n$ of all binary strings of length $n$. For any vertex $u$ we denote by $u_k$ its $k$-th digit. Then the probability that a pair of vertices $\{ u, v\}$ are connected by an edge is 
\[p_{u,v}=\prod_{k=1}^n P_{u_k,v_k}\]
 independently of the presence or absence of any other edge. Without loss of generality we may assume that $\gamma\leq \alpha$.

The stochastic Kronecker graph extends the binomial random graph model $G_{2^n\!,p}$, where every edge is inserted with probability $p$, independently of the presence or absence of any other edge, because selecting $\alpha=\beta=\gamma$ in the stochastic Kronecker graph ensures that every edge is inserted with the same probability $p=\alpha^n$.

Stochastic Kronecker graphs have been considered when $\alpha, \beta, \gamma$ are fixed constants not depending on $n$. Mahdian and Xu \cite{MR2829312} considered the graph when $\alpha\geq \beta \geq \gamma$. 
They have shown that the diameter of the stochastic Kronecker graph is bounded from above by a constant when it is connected. The appearance of the giant component in this graph has also been investigated. Horn and Radcliffe \cite{MR2900144} extending the result of Mahdian and Xu \cite{MR2829312} showed that $(\alpha+\beta)(\beta+\gamma)>1$ is a necessary and sufficient condition for the appearance of a unique giant component. They also determined the number of vertices in the giant component. Radcliffe and Young \cite{Kroneckerconnectivity} analyzed the connectivity and the size of the giant component in a generalized version of the stochastic Kronecker graph. Their results imply that the threshold for connectivity in stochastic Kronecker graphs is $\beta+\gamma=1$.

Leskovec, Chakrabarti, Kleinberg and Faloutsos \cite{Kroneckerintr} have shown empirically that stochastic Kronecker graphs resemble several real world networks and claimed that the model exhibits a {power law degree distribution}. Later, Leskovec, Chakrabarti, Kleinberg, Faloutsos and Ghahramani \cite{Kroneckerfit} fitted the model to several real world networks such as the Internet, citation graphs and online social networks.

The R-MAT model, introduced by Chakrabarti, Zhan and Faloutsos \cite{R-MATintr}, is closely related. The vertex set of this model is also $\mathbb{Z}_2^n$ and one also has parameters $\alpha, \beta, \gamma$. However, in this case one needs the additional condition that $\alpha+2\beta+\gamma=1$. In this model one generates $m$ vertex pairs $({u}, {v})$ in such a way that 
\[\mathbb{P}\Big(\left(u_k,v_k\right)=(a,b)\Big)=\left\{
	\begin{array}{ll}
		\alpha  & \mbox{if } a=b=1, \\
		\gamma & \mbox{if } a=b=0,\\
		\beta & \mbox{else},
	\end{array}
\right.
\]
 independently for each digit and inserts an edge between every generated vertex pair. The process creates a multigraph with $m$ edges and the graph after the multi-edges have been merged is refered to as the R-MAT model. The advantage of the R-MAT model over the stochastic Kronecker graph is that it can be generated significantly faster when $m$ is small. The degree sequence of this model has been studied by Gro\"er, Sullivan and Poole \cite{MR2836474} and by Seshadhri, Pinar and Kolda \cite{SPK11} when $m=\Theta(2^n)$, i.e.\ the number of edges is linear in the number of vertices. They have shown that the degree sequence of the model does not follow a power law distribution. However, no rigorous proof exists for the equivalence of the two models and in the stochastic Kronecker graph there is no restriction on the sum of the values of $\alpha,\beta,\gamma$.

\subsection{Main results}

In this paper we examine the asymptotic behavior of the stochastic Kronecker graph $K(n,P)$, when the entries of $P$ are fixed constants (independent of $n$). A property $\mathcal{P}$ holds asymptotically almost surely (in short a.a.s.) if the probability that $\mathcal{P}$ holds tends to one as $n\rightarrow \infty$. Furthermore we ignore floors and ceilings. The real world networks modeled with the stochastic Kronecker graphs are claimed to have a {\it power law degree distribution}. We  show that this does not match the asymptotic behavior of the stochastic Kronecker graph, which a.a.s.\ does not follow a power law degree distribution.
\begin{theorem}\label{main}
For all parameters $0<\alpha,\beta,\gamma<1$ the stochastic Kronecker graph $K(n,P)$ a.a.s.\ does not have a power law degree distribution.
\end{theorem}

Recall that that the binomial random graph $G_{2^n\! ,\alpha^n}$ is a special case of the stochastic Kronecker graph $K(n,P)$ when $\alpha=\beta=\gamma$. The question concerning the degree distribution and the number of subgraphs in the binomial random graph have been thoroughly studied. 

Erd\H{o}s and R\'enyi \cite{MR0125031} showed that the degree distribution of $G_{2^n\!,\alpha^n}$ follows a Poisson distribution when $\alpha\leq 1/2$, where the parameter of the distribution depends on $\alpha$. They also showed that if $\alpha>1/2$, then there are a.a.s.\ no vertices of finite degree, in fact the degree of every vertex is a.a.s.\ $(1+o(1))(2\alpha)^n$.

The threshold for the appearance of subgraphs for $G_{2^n\! ,\alpha^n}$ has been established by Bollob\'as \cite{MR620729}. Let $G$ be a small graph and let $v(G)$ and $e(G)$ denote the number of vertices and edges of $G$. The threshold for the appearance of $G$  is the smallest value of $\alpha$ such that $2^{v(G')}\alpha^{e(G')}\geq 1$ holds for every $G'\subseteq G$. Additionally Alon and Spencer \cite{probabilisticmethod} show that if for every $G'\subseteq G$ we have that $2^{v(G')}\alpha^{e(G')}>1$, then a.a.s.\ there are $(1+o(1))\left(2^{v(G)}\alpha^{e(G)}\right)^n$ labeled copies of $G$ present in $G_{2^n\!,\alpha^n}$.

We examine the subgraphs contained in the stochastic Kronecker graph. The first result determines the expected number of copies of a given subgraph present in $K(n,P)$.

\begin{lemma}\label{subgraphs}
Let $G$ be a simple graph, let $X_G$ be the number of labeled copies of $G$ in $K(n,P)$ and let $L_G$ be the set of functions $g:V(G)\rightarrow \mathbb{Z}_2$.
Then we have
\[\mathbb{E}(X_G)=(1+o(1))\left(\sum_{g\in L_G}\prod_{\{u,v\}\in E(G)}P_{g(u),g(v)}\right)^n.\]
\end{lemma}

We also show concentration for several classes of graphs. In particular, when $\alpha=\gamma,$ the number of copies of a cycle of length $k$ contained in $K(n,P)$ is concentrated around its mean. 

\begin{theorem}\label{cycleconc}
Let $C_k$ be a cycle of length $k$ and $X_{C_k}$ the number of labeled copies of $C_k$ in $K(n,P)$. Assume that $\alpha=\gamma$. Then the threshold for the appearance of $C_k$ in the stochastic Kronecker graph is $(\alpha+\beta)^{k}+(\alpha-\beta)^k=1$. Additionally, if $(\alpha+\beta)^{k}+(\alpha-\beta)^k>1$, then a.a.s.\ $$X_{C_k}=(1+o(1))\left((\alpha+\beta)^{k}+(\alpha-\beta)^k\right)^n.$$
\end{theorem}

Theorem~\ref{cycleconc} implies that the even cycles appear in the stochastic Kronecker graph in order of their length (or at the same time, when $\alpha=\beta$). More precisely, one can find parameters $\alpha,\beta$ such that $C_{2k-2}$ is a.a.s.\ present in the stochastic Kronecker graph but $C_{2k}$ is not, and for any parameter $\alpha,\beta$ for which $C_{2k}$ is a.a.s.\ present in the stochastic Kronecker graph then so is $C_{2k-2}$. The same holds for odd cycles when $\alpha>\beta$, but the reverse is true when $\beta>\alpha$ as in this case $C_{2k+1}$ appears before $C_{2k-1}$. This is due to the fact that in the stochastic Kronecker graph, when $\alpha=\gamma$, the neighborhood of every vertex consists mostly of vertices that differ on approximately $\beta n/(\alpha+\beta)$ digits. In order to state this more formally denote by $N(u)$ the neighborhood of a vertex $u$ in $K(n,P)$ and for vertices $u,v$ in $K(n,P)$ let $H(u,v)$ be the Hamming distance between $u$ and $v$. 
Amongst other results we prove the following theorem in Section \ref{structure}.

\begin{theorem}\label{ndistance}
 Assume that $\alpha=\gamma$ and $\alpha+\beta>1$. Then a.a.s.\ for all vertices $u$ in $K(n,P)$ we have that 
$|N(u)|=(1+o(1))(\alpha+\beta)^n$ and
\[\left|\left\{w\in N(u): H(u,w)=(1+o(1)) \frac{\beta}{\alpha+\beta}n\right\}\right|=(1+o(1))(\alpha+\beta)^n.\]
\end{theorem}

\subsection{Outline of the proofs}
The proof of Theorem \ref{main} relies on calculating the expected number of vertices of degree $d$. The probability that a vertex has degree $d$ depends only on its \emph{weight}, i.e.\ the sum of its digits. To be more precise, the degree of a vertex with weight $w$ is a multinomial random variable, however it can be approximated by a Poisson random variable with parameter $(\alpha+\beta)^w(\beta+\gamma)^{n-w}$. Therefore, the expected number of vertices of degree $d$ is approximately
\[\sum_{w=0}^n \binom{n}{w} \frac{(\alpha+\beta)^{dw}(\beta+\gamma)^{d(n-w)}}{d!}\exp(-(\alpha+\beta)^w(\beta+\gamma)^{n-w}).\]
The real difficulty in proving Theorem \ref{main} lies in determining the value of this sum. It turns out that the parameters $\alpha,\beta,\gamma$ have to satisfy $\alpha+\beta=\beta+\gamma=1$ for this sum to be $\Theta(2^n)$ for every finite value of $d$. However this would be a necessary condition for the degree sequence of a graph to follow a power law distribution. Therefore the parameters $\alpha,\beta,\gamma$ must satisfy $\alpha+\beta=\beta+\gamma=1$ for the graph to have a power law degree sequence. However, in this case the sum simplifies to
\[\sum_{w=0}^n\binom{n}{w}\frac{1}{\mathrm{e}d!}=2^n\frac{1}{\mathrm{e}d!},\]
which indicates that the degree sequence follows a Poisson distribution with parameter 1, not a power law.

In order to show concentration for subgraphs we use the second moment method. As in Lemma~\ref{subgraphs}, let $X_G$ denote the number of labeled copies of $G$ in $K(n,P)$. Then $X_G$ is concentrated if we can show that $\mathbb{E}(X_F)=o\left((\mathbb{E}(X_G))^2\right)$ for every graph $F$ which is the union of two edge-overlapping copies of $G$. 
The major difficulty to overcome is that although Lemma \ref{subgraphs} gives us a formula for calculating the expected number of labeled copies of a graph it does not give us a closed formula or even a simple method to compare the expected number of copies of two different graphs. 
We examine classes of graphs where we can express $\mathbb{E}(X_G)$ in a closed form. However, it is still difficult to compare $(\mathbb{E}(X_G))^2$ to $\mathbb{E}(X_F)$ for most graphs $F$ formed of two edge-overlapping copies of $G$. It turns out that for the classes of graphs considered in this paper $\mathbb{E}(X_F)$ takes its maximum when the two copies of $G$ overlap in as many edges as possible or as few as possible. These graphs resemble either two disjoint copies of $G$ or a single copy of $G$, enabling us to compare the expected number of copies from these graphs.

Finally, we examine the neighborhood of the vertices in the stochastic Kronecker graph when $\alpha=\gamma$ and $\alpha+\beta>1$. Fix a vertex $v$. Under the conditions we have that the expected number of neighbors that differ on precisely $k$ elements from $v$ is 
$$\binom{n}{k}\alpha^{n-k}\beta^k=(\alpha+\beta)^n\, \mathbb{P}\left(\mathrm{Bin}\left(n,\frac{\beta}{\alpha+\beta}\right)=k\right),$$
which implies, by summing over all $k\in\{0,\dots,n\},$ that the expected degree of $v$ is $(\alpha+\beta)^n$. Moreover it is a well-known fact that any binomial random variable is concentrated around its mean. Therefore we would expect that almost all of the contribution to the degree of $v$ comes from the terms where $k\approx \beta n/(\alpha+\beta)$. In fact this already shows us that this holds in expectation and one can show that this holds a.a.s.\ using Chernoff's inequalities. 

\section{Degree sequence: proof of Theorem \ref{main}}\label{degree}
In order to establish the degree sequence of the stochastic Kronecker graph we first need to determine the expectation and the variance of the degree of a fixed vertex. 
\begin{lemma} \label{expectation}
Let $u,v$ be two vertices in $K(n,P).$ We denote by $d(v)$ the degree of $v$ and by $w(v)=\sum_{k=1}^{n}v_k$ its weight. Furthermore, let $I_{u,v}$ be the event that the edge $\{u,v\}$ is present in the graph. Then
\begin{align*}
\mathbb{E}(d(v))&=(\alpha+\beta)^{w(v)}(\beta+\gamma)^{n-w(v)},\\
\sum_{u\in \mathbb{Z}_2^n} (\mathbb{E}(I_{u,v}))^2&=(\alpha^2+\beta^2)^{w(v)}(\beta^2+\gamma^2)^{n-w(v)},\\
\mathrm{Var}(d(v))&=(\alpha+\beta)^{w(v)}(\beta+\gamma)^{n-w(v)}-(\alpha^2+\beta^2)^{w(v)}(\beta^2+\gamma^2)^{n-w(v)}.
\end{align*}
\end{lemma}

\begin{proof}
We assume $u$ can be created from $v$ by changing $i$ ones to zeros and $j$ zeros to ones. Then the probability that the edge $\{u,v\}$ is present is $\alpha^{w(v)-i} \beta^{i}\gamma^{n-w(v)-j}\beta^{j}$. Thus we have
\begin{align*}
\mathbb{E}(d(v))&=\sum_{i=0}^{w(v)}\sum_{j=0}^{n-w(v)}\binom{w(v)}{i}\binom{n-w(v)}{j}\alpha^{w(v)-i} \beta^{i}\gamma^{n-w(v)-j}\beta^{j}\\
&=\sum_{i=0}^{w(v)}\binom{w(v)}{i}\alpha^{w(v)-i}\beta^{i}\sum_{j=0}^{n-w(v)}\binom{n-w(v)}{j}\gamma^{n-w(v)-j}\beta^{j}\\
&=(\alpha+\beta)^{w(v)}(\beta+\gamma)^{n-w(v)}.
\end{align*}
Similarly, we get
\begin{align*}
\sum_{u\in \mathbb{Z}_2^n} &(\mathbb{E}(I_{u,v}))^2=\sum_{i=0}^{w(v)}\sum_{j=0}^{n-w(v)}\binom{w(v)}{i}\binom{n-w(v)}{j}\left(\alpha^{w(v)-i} \beta^{i}\gamma^{n-w(v)-j}\beta^{j}\right)^2\\
&=\sum_{i=0}^{w(v)}\binom{w(v)}{i}(\alpha^2)^{w(v)-i}(\beta^2)^{i}\sum_{j=0}^{n-w(v)}\binom{n-w(v)}{j}(\gamma^2)^{n-w(v)-j}(\beta^2)^{j}\\
&=(\alpha^2+\beta^2)^{w(v)}(\beta^2+\gamma^2)^{n-w(v)},
\end{align*}
and these results together imply
\begin{align*}
\mathrm{Var}(d(v))&=\sum_{u\in \mathbb{Z}_2^n}\mathrm{Var}(I_{u,v})=\sum_{u\in \mathbb{Z}_2^n} \left[\mathbb{E}(I_{u,v})-(\mathbb{E}(I_{u,v}))^2\right]\\
&=(\alpha+\beta)^{w(v)}(\beta+\gamma)^{n-w(v)}-(\alpha^2+\beta^2)^{w(v)}(\beta^2+\gamma^2)^{n-w(v)}. \, \qedhere
\end{align*}
\end{proof}

In order to better understand the behavior of the random variable given by the number of vertices of a certain degree in $K(n,P)$ we first examine the distribution of the degree of a fixed vertex. We show a normal or a Poisson approximation for the degree of a given vertex depending on its expected degree. The existence of a normal approximation indicates that the vertices with high expected degree are unlikely to have constant (independent of $n$) degree.

Let $S_1,S_2,...$ and $Z$ be random variables. We say that the sequence $S_n$ converges in distribution to $Z$ as $n\rightarrow \infty$, denoted by $S_n\xrightarrow{d}Z$, if \linebreak[4]$ \mathbb{P}(S_n\leq x)\rightarrow \mathbb{P}(Z\leq x)$ for every real $x$ that is a continuity point of $\mathbb{P}(Z\leq x)$. For the normal approximation of the vertex degree we apply the Chen-Stein method in the form of the following theorem, which is a simplified version of Theorem 6.33 in Janson, {\L}uczak and Ruci{\'n}ski \cite{MR1782847}.

\begin{theorem}[Chen-Stein method \cite{MR1782847}] \label{steinnormal}
Suppose that $(S_{n})_1^{\infty}$ is a sequence of random variables such that $S_n=\sum_{a\in A_n}X_{n,a}$, where for each $n$, \linebreak[4] $\{X_{n,a}\}_{a\in A_n}$ is a family of mutually independent indicator random variables. If
\[\frac{\mathbb{E}(S_n)}{(\mathrm{Var}(S_n))^{3/2}}\rightarrow 0,\]
then
\[\frac{{S}_n-\mathbb{E}(S_n)}{\mathrm{Var}(S_n)}\xrightarrow{d} N(0,1).\]
\end{theorem}

Using Lemma \ref{expectation} and Theorem \ref{steinnormal} we obtain the following normal approximation for the degree of a vertex in $K(n,P).$

\begin{lemma}
For any fixed vertex $v\in V(K(n,P))$ with $\mathbb{E}\left(d\left(v\right)\right)\to\infty$ we have 
\[\frac{d(v)-\mathbb{E}(d(v))}{\mathrm{Var}(d(v))}\xrightarrow{d} N(0,1).\]
\end{lemma}

\begin{proof}
Note that
\[\frac{(\alpha^2+\beta^2)^{w(v)}(\beta^2+\gamma^2)^{n-w(v)}}{(\alpha+\beta)^{w(v)}(\beta+\gamma)^{n-w(v)}}=o(1)\]
as $\alpha^2+\beta^2<\alpha+\beta$ and $\beta^2+\gamma^2<\beta+\gamma$ and at least one of  $w(v)$ and $n-w(v)$ tends to infinity. Therefore, by Lemma~\ref{expectation}, we have $\mathrm{Var}(d(v))=(1+o(1))\mathbb{E}(d(v)).$ Thus the conditions of Theorem~\ref{steinnormal} are satisfied and the statement follows.
\end{proof}

In order to determine whether the stochastic Kronecker graph can have a power law degree distribution we need to consider the number of vertices of a given fixed degree. The following lemma provides a Poisson approximation for such a random variable if some conditions on its first and second moment are satisfied.

\begin{lemma}\label{poissondistr}
Let $(X_{n})_1^{\infty}$ be a sequence of random variables such that $X_n=\sum_{a\in A_n}I_{a,n}$, where for each $n$ the $I_{a,n}$'s are mutually independent indicator random variables. Define $\lambda_n=\mathbb{E}(X_n)$.  Further assume we have that the following conditions are satisfied:
\begin{align*}
\sum_{a\in A_n} \mathbb{P}(I_{a,n}=1)^2&=o(\lambda_n^2),\\
\max_{a\in A_n}\mathbb{P}(I_{a,n}=1)&=o(1),\\
\lambda_n\max_{a\in A_n}\mathbb{P}(I_{a,n}=1)&=o(1).
\end{align*}
Then for every finite $k$ we have 
\[\mathbb{P}(X_n=k)=(1+o(1))\mathbb{P}\left(\mathrm{Po}\left(\lambda_n\right)=k\right).\]
\end{lemma}

\begin{proof}
By the representation of $X_n$ as the sum of the $I_{a,n}$'s we get
\begin{align}
\mathbb{P}(X_n=k)&=\sum_{\substack{A'\subseteq A_n\\ |A'|=k}}\left(\prod_{a\in A'}\mathbb{P}(I_{a,n}=1)\right)\left(\prod_{a\in A_n\backslash A'}\mathbb{P}(I_{a,n}=0)\right)\\
\nonumber&\hspace{-1.2cm}=(1+o(1))\sum_{\substack{A'\subseteq A_n \\ |A'|=k}}\left(\prod_{a\in A'}\mathbb{P}(I_{a,n}=1)\right)\left(\prod_{a\in A_n}(1-\mathbb{P}(I_{a,n}=1))\right),
\end{align}
since $\mathbb{P}\left(I_{a,n}=1\right)\leq \max_{a'\in A_n}\mathbb{P}(I_{a',n}=1) =o(1)$ and $|A'|$ is finite.

Using the standard estimate $\exp(-x)\geq 1-x$ which holds for any $x\in\mathbb{R}$ and the fact that $1-x~\geq~\exp(-x(1+x))$ for any $x\in[0,(\sqrt{5}-1)/2],$ we have
\begin{align*}
\exp(-\lambda_n)\geq\prod_{a\in A_n}(1-\mathbb{P}(I_{a,n}=1))\geq \exp\left(-\lambda_n\left(1+\max_{a\in A_n}\mathbb{P}\left(I_{a,n}=1\right)\right)\right).
\end{align*}
and furthermore the upper and lower bound coincide asymptotically due to the condition $\lambda_n\max_{a\in A_n}\mathbb{P}\left(I_{a,n}=1\right)=o(1)$. Hence, we have 
\begin{align}
\mathbb{P}\left(X_n=k\right)&= (1+o(1))\exp(-\lambda_n)\, S_{k,A_n}\, ,\label{eq:poissondistr}
\end{align}
where we abbreviate the sum over all subsets of $A_n$ with size $k$ by $$S_{k,A_n}=\sum_{\substack{A'\subseteq A_n \\ |A'|=k}}\prod_{a\in A'}\mathbb{P}(I_{a,n}=1).$$

It remains to establish the asymptotic behavior of $S_{k,A_n}.$ First we obtain an upper bound by summing over all multi-sets of size $k$ and applying the Multinomial Theorem
\begin{align*}
S_{k,A_n}&\leq \left.\left(\sum_{a\in A_n} \mathbb{P}(I_{a,n}=1)\right)^k\right/k!= \frac{\lambda_n^k}{k!}.
\end{align*}
But this upper bound is asymptotically tight as seen by the following argument. Since we only added summands for multi-sets that have at least one repetition we obtain the following upper bound for the difference
\begin{align*}
\frac{\lambda_n^k}{k!}-S_{k,A_n}&\leq \sum_{a\in A_n} \mathbb{P}(I_{a,n}=1)^2\left(\sum_{a\in A_n} \mathbb{P}(I_{a,n}=1)\right)^{k-2}\\
&=\lambda_n^{k-2}\sum_{a\in A_n} \mathbb{P}(I_{a,n}=1)^2,
\end{align*}
and note that this is $o(\lambda_n^k)$ since $\sum_{a\in A_n} \mathbb{P}(I_{a,n}=1)^2=o(\lambda_n^2).$ Hence, by Equation \eqref{eq:poissondistr} we have 
\begin{align*}
\mathbb{P}\left(X_n=k\right)&= (1+o(1))\exp(-\lambda_n)\frac{\lambda_n^k}{k!}
\end{align*}
as claimed, completing the proof.
{\ }
\end{proof}

In particular applying Lemma~\ref{poissondistr} to the stochastic Kronecker graph $K(n,P)$ provides a formula for the expected number of vertices of a given fixed degree.

\begin{lemma} \label{ExpFormula}
Fix $d\in\mathbb{N}$ and let $\nvd_d$ denote the number of vertices of degree $d$ in the stochastic Kronecker graph $K(n,P).$  Then we have
\begin{equation*}\mathbb{E}(\nvd_d)=(1+o(1))\sum_{w=0}^n \binom{n}{w} \frac{(\alpha+\beta)^{dw}(\beta+\gamma)^{d(n-w)}}{d!}e^{-(\alpha+\beta)^w(\beta+\gamma)^{n-w}}+o(1).\end{equation*}
\end{lemma}

\begin{proof}
Let $\eta=\max\{\alpha,\beta\}$ and note that  $\eta^{-n}$ grows exponentially with $n,$ since $\eta<1$. Furthermore, let $w_0$ be the maximal weight $w\in\{0,\dots,n\}$ such that \begin{equation}\label{extrWeight}
(\alpha+\beta)^{w}(\beta+\gamma)^{n-w}<\eta^{-n}/\log{n}.
\end{equation} In case no such $w$ exists, we set $w_0=-1$. We denote by ${\nvd}_{d,w_0}$ the number of vertices of degree $d$ which have weight at most $w_0$. Now consider any vertex $v$ in $K(n,P)$ such that $w(v)\leq w_0$. First note that there is an edge between vertices $u$ and $v$ with probability at most $\eta^n=\max\{\alpha^n,\beta^n\}=o(1)$, uniformly for all such vertex pairs. Moreover since $\mathbb{E}\left(d(v)\right)<\eta^{-n}/\log{n},$ by Lemma~\ref{expectation} and Inequality \eqref{extrWeight}, we get
$$\mathbb{E}(d(v))\, p_{u,v}=o(1),$$ for any vertex $u\in\mathbb{Z}_2^n$.
Moreover, by Lemma~\ref{expectation}, we have that
\begin{align}
\sum_{u\in \mathbb{Z}_2^n} \mathbb{P}(I_{u,v}=1)^2&=(\alpha^2+\beta^2)^{w(v)}(\beta^2+\gamma^2)^{n-w(v)}\\
\nonumber &\hspace{-0.5cm}=o\left(\left(\alpha+\beta\right)^{2w(v)}\left(\beta+\gamma\right)^{2(n-w(v))}\right)=o\left(\left(\mathbb{E}\left(d\left(v\right)\right)\right)^2\right).\label{Var}
\end{align}
Hence, we can apply Lemma~\ref{poissondistr} to the summands of $\nvd_{d,w_0}$ and obtain
\[\mathbb{E}({\nvd}_{d,w_0})=(1+o(1))\sum_{w=0}^{w_0} \binom{n}{w} \frac{(\alpha+\beta)^{dw}(\beta+\gamma)^{d(n-w)}}{d!}e^{-(\alpha+\beta)^w(\beta+\gamma)^{n-w}}.\]

Now for any vertex $v$ in $K(n,P)$ with $w(v)>w_0$, Inequality \eqref{extrWeight} does not hold, which implies that the tail
\begin{align*}
\sum_{w=w_0+1}^n \binom{n}{w}& \frac{(\alpha+\beta)^{dw}(\beta+\gamma)^{d(n-w)}}{d!}\exp\left(-(\alpha+\beta)^w(\beta+\gamma)^{n-w}\right) 
\end{align*}
is dominated by the exponential term and hence $o(1).$ Therefore, summing up to $n$ instead of $w_0$ implies only an additive error of order $o(1).$

To finish the proof we have to consider the contribution of vertices of large weight. Let $v$ be a vertex with $w(v)> w_0$ and observe that by the definition of $w_0$ we have $\mathbb{E}\left(d\left(v\right)\right)\to\infty$ exponentially.
Therefore, Chernoff's inequality yields that 
$$\mathbb{P}\left(d\left(v\right)=d\right)\leq \exp\left(-\frac{(1+o(1))\left(\mathbb{E}\left(d\left(v\right)\right)\right)^2}{2\left[\mathrm{Var}(d\left(v\right))+\mathbb{E}\left(d\left(v\right)\right)/3\right]}\right) $$ and furthermore, by Lemma~\ref{expectation}, this implies $$\mathbb{P}\left(d\left(v\right)=d\right)\leq\exp\left(-\mathbb{E}\left(d\left(v\right)\right)/3\right)=o\left(e^{-n}\right).$$ 
Consequently the expected number of vertices of degree $d$ with weight larger than $w_0$ is also $o(1)$ and the statement follows. 
\end{proof}

Next we prove an auxiliary lemma that will be used frequently in the remainder of this section.

\begin{lemma}\label{binomappr}
Let $x,y>0$. For any $c<x/(x+y)$ and any $w_0<cn$ 
we have  
\[\sum_{w=0}^{w_0}\binom{n}{w}x^w y^{n-w}=\Theta\left(\binom{n}{w_0}x^{w_0}y^{n-w_0}\right),\] and likewise, for any $c>x/(x+y)=1-y/(x+y)$ and any $w_0>cn,$ we have \[\sum_{w=w_0}^{n}\binom{n}{w}x^w y^{n-w}=\Theta\left(\binom{n}{w_0}x^{w_0}y^{n-w_0}\right). \]
\end{lemma}
\begin{proof}
Due to symmetry we only consider the first statement. First observe that the following elementary inequalities hold for all $1\leq j\leq w_0$ 
\begin{equation}
0<\frac{w_0-j+1}{n-w_0+j}\leq\frac{w_0}{n-w_0+1}<\frac{c}{(1-c)},\label{tailIneq}
\end{equation}
since $w_0<cn.$ But this already implies that
\begin{align*}
\sum_{w=0}^{w_0}\binom{n}{w}x^w y^{n-w}&=\binom{n}{w_0}x^{w_0} y^{n-w_0}\left(\sum_{i=0}^{w_0}\prod_{j=1}^i\frac{w_0-j+1}{n-w_0+j}\frac{y}{x}\right)\\
&\stackrel{\eqref{tailIneq}}{\leq} \binom{n}{w_0}x^{w_0} y^{n-w_0}\left(\sum_{i=0}^{w_0}\left(\frac{c}{1-c}\frac{y}{x}\right)^i\right)\\
&=O\left(\binom{n}{w_0}x^{w_0} y^{n-w_0}\right),
\end{align*}
where the asymptotic statement holds since $c<\frac{x}{x+y}$ implies $\frac{c}{1-c}\frac{y}{x}<1$ and therefore the sum is a partial sum of a convergent geometric series. The proof follows from the fact that 
\begin{equation*}
\sum_{w=0}^{w_0}\binom{n}{w}x^w y^{n-w}>\binom{n}{w_0}x^{w_0}y^{n-w_0}.\qedhere
\end{equation*}

\end{proof}

In order to prove Theorem~\ref{main} we need one more lemma calculating the asymptotic value of the expected number of vertices in the stochastic Kronecker graph with a fixed degree.

\begin{lemma}\label{lem:expdeg}
Let $0<\alpha,\beta,\gamma<1$ be arbitrary parameters of $K(n,P).$ For any fixed $d\in\mathbb{N}$ denote by $\nvd_d$ the number of vertices of degree $d$ in $K(n,P)$ as in Lemma~\ref{ExpFormula}. Then we have either \begin{equation}\label{stmt1}
\mathbb{E}(\nvd_d)=\Theta\left(\left(\left(\alpha+\beta\right)^d+\left(\beta+\gamma\right)^d\right)^n\right)
\end{equation} or \begin{equation}\label{stmt2}
\mathbb{E}(\nvd_d)=o(2^n).
\end{equation} 
\end{lemma}

\begin{proof}
There are six cases, according to the choice of $\alpha,$ $\beta$ and $\gamma$, that require different calculations depending on the terms that dominate the expectation of the number of vertices of degree $d$ in $K(n,P).$ For this we will use the asymptotic representation of $\mathbb{E}\left(\nvd_d\right)$ given in Lemma~\ref{ExpFormula}.\\

In order to shorten our notation we set
\[a_{d,w}=\binom{n}{w} \frac{(\alpha+\beta)^{dw}(\beta+\gamma)^{d(n-w)}}{d!}\exp(-(\alpha+\beta)^w(\beta+\gamma)^{n-w})\]
 and we obtain the representation \begin{equation}
\label{ExpDeg} \mathbb{E}\left(\nvd_d\right)=(1+o(1))\sum_{w=0}^n a_{d,w}+o(1).
 \end{equation}

\bf Case 1:\hspace{.2cm} $\beta+\gamma<\alpha+\beta=1.$\normalfont \\
Note that in this case
\[a_{d,w}=\binom{n}{w} \frac{(\beta+\gamma)^{d(n-w)}}{d!}\exp(-(\beta+\gamma)^{n-w}).\]
Furthermore, since $\exp\left(-\left(\beta+\gamma\right)^{n-w}\right)\leq 1$, for every $0\leq w \leq n$, we get
\begin{align*}
\mathbb{E}(Z_d)\stackrel{\eqref{ExpDeg}}{=}(1+o(1))\sum_{w=0}^n a_{d,w}+o(1)&\leq(1+o(1))\frac{\left(1+\left(\beta+\gamma\right)^d\right)^n}{d!}\, . 
\end{align*}
On the other hand, there is a constant $\varepsilon\in \big(0,(\beta+\gamma)/(1+\beta+\gamma)\big)$ such that for all $w\leq(1-\varepsilon)n$ we have $\exp\left(-\left(\beta+\gamma\right)^{n-w}\right)= 1+o(1)$  and thus 
\begin{align*}
\mathbb{E}(Z_d)&\stackrel{\eqref{ExpDeg}}{=}(1+o(1))\sum_{w=0}^n a_{d,w}+o(1)\\
&\geq(1+o(1))\sum_{w=0}^{(1-\varepsilon)n}\binom{n}{w}\frac{(\beta+\gamma)^{d(n-w)}}{d!}+o(1)\\
&\stackrel{L.\ref{binomappr}}{=}(1+o(1))\frac{\left(1+\left(\beta+\gamma\right)^d\right)^n}{d!}
\end{align*}
and Statement \eqref{stmt1} holds.\\

\bf Case 2:\hspace{.2cm} $1=\beta+\gamma<\alpha+\beta.$\normalfont \\
Let us first introduce two parameters: 
$$K=\log{n}-\log\log_{\alpha+\beta}\log{n}+(d+1)\log(\alpha+\beta),$$
$$k=\log_{\alpha+\beta}\left(\frac{K}{(\alpha+\beta)-1}\right) .$$
Observe that, asymptotically, we have \begin{equation}
\label{asympK}K=(1+o(1))\log n,
\end{equation} and thus \begin{equation}
\label{asympk}k=\log_{\alpha+\beta}\log n +O(1).
\end{equation}
 Next, note that in this case for any weight $w$ we have $$a_{d,w}=\binom{n}{w}\frac{\left(\alpha+\beta\right)^{dw}}{d!}\exp\left(-\left(\alpha+\beta\right)^w\right),$$ and thus, using Estimates \eqref{asympK} and \eqref{asympk}, we get  
 $$a_{d,k}=o\left(2^nn^{-(1+o(1))}(\log n)^d \right)=o(2^n).$$ Therefore Statement \eqref{stmt2} follows if we show that $\mathbb{E}(\nvd_d)=\Theta(a_{d,k}).$
In fact, it is sufficient to show that 
\begin{equation}\label{sumCase5}
\sum_{w=0}^{n}\frac{a_{d,w}}{a_{d,k}}=O(1),\end{equation}
since $\mathbb{E}(\nvd_d)\geq(1+o(1)) a_{d,k}+o(1)$ and $a_{d,k}=\omega(1)$.
We can divide this sum into three parts
\[\sum_{w=0}^{n}\frac{a_{d,w}}{a_{d,k}}=\sum_{w=1}^{k}\frac{a_{d,k-w}}{a_{d,k}}+1+\sum_{w=1}^{n-k}\frac{a_{d,k+w}}{a_{d,k}}\]
and analyze the parts of this sum separately. By the definition of $a_{d,w}$ we get for the first part 
\begin{align}
\nonumber \sum_{w=1}^{k}\frac{a_{d,k-w}}{a_{d,k}}&\leq(1+o(1))\sum_{w=1}^{k}\left(\frac{k}{n}\right)^w(\alpha+\beta)^{-dw} e^{\left(\alpha+\beta\right)^k\left(1-\left(\alpha+\beta\right)^{-w}\right)}.
\end{align}Furthermore, by the definition of $k,$ we have
\begin{align*}
\left(\alpha+\beta\right)^k\left(1-\left(\alpha+\beta\right)^{-w}\right)&=\frac{K}{\alpha+\beta}\cdot\frac{1-\left(\alpha+\beta\right)^{-w}}{1-\left(\alpha+\beta\right)^{-1}}\\
&=\frac{K}{\alpha+\beta}\sum_{i=0}^{w-1}\left(\alpha+\beta\right)^{-i}\\
&\stackrel{\alpha+\beta>1}{\leq} \frac{Kw}{\alpha+\beta},
\end{align*}
 and thus, extending the range of summation, we obtain \begin{align*}
\sum_{w=1}^{k}\frac{a_{d,k-w}}{a_{d,k}}&\leq(1+o(1))\sum_{w=1}^{\infty}\left(\frac{k}{(\alpha+\beta)^{d}n}\exp\left(\frac{K}{\alpha+\beta}\right)\right)^w.
\end{align*} Substituting $K$ and using Estimate \eqref{asympk} for $k$ yields \begin{align*}
\frac{k}{(\alpha+\beta)^{d}n}\exp\left(\frac{K}{\alpha+\beta}\right)&=\frac{\log_{\alpha+\beta}\log{n}+O(1)}{(\alpha+\beta)^{d}n}\left(\frac{n(\alpha+\beta)^{d+1}}{\log_{\alpha+\beta}\log{n}}\right)^{1/(\alpha+\beta)}\\
&=O\left(n^{-\frac{\alpha+\beta-1}{\alpha+\beta}}\log n\right) ,
\end{align*} and thus the first part of the sum in \eqref{sumCase5} is bounded by $$\sum_{w=1}^{k}\frac{a_{d,k-w}}{a_{d,k}}
\leq\sum_{w=1}^\infty \left(O\left(n^{-\frac{\alpha+\beta-1}{\alpha+\beta}}\log n\right)\right)^w=O(1).$$

Now it remains to consider the last term in the sum in \eqref{sumCase5}. The arguments involved in this step are similar to the ones just used, but we include them for sake of completeness. By the definition of $a_{d,w}$ we get
\begin{align*}
\sum_{w=1}^{n-k}\frac{a_{d,k+w}}{a_{d,k}}&\leq (1+o(1))\sum_{w=1}^{n-k}\left(\frac{n}{k}\right)^w(\alpha+\beta)^{dw}e^{-\left(\alpha+\beta\right)^k\left(\left(\alpha+\beta\right)^w-1\right)}.
\end{align*} Furthermore, by the definition of $k,$ we estimate 
\begin{align*}
-\left(\alpha+\beta\right)^k\left(\left(\alpha+\beta\right)^{w}-1\right)&=-K\frac{\left(\alpha+\beta\right)^{w}-1}{\alpha+\beta-1}\\
&=-K\sum_{i=0}^{w-1}\left(\alpha+\beta\right)^{i}\\
&\stackrel{\alpha+\beta>1}{\leq} -Kw.
\end{align*}
Thus, extending the range of summation, substituting $K$, and using Estimate \eqref{asympk} for $k$ shows that \begin{align*}
\sum_{w=1}^{n-k}\frac{a_{d,k+w}}{a_{d,k}}&\leq(1+o(1))\sum_{w=1}^{\infty}\left(\frac{n(\alpha+\beta)^d}{k}\exp\left(-K\right)\right)^w=\sum_{w=1}^{\infty} \left(\frac{1+o(1)}{\alpha+\beta}\right)^w.
\end{align*}
Since $\alpha+\beta>1$ the last term is $O(1)$ and, as stated earlier, Statement \eqref{stmt2} holds.\\

\bf Case 3:\hspace{.2cm} $\beta+\gamma<1<\alpha+\beta.$\normalfont \\
We start by defining two constants $c_1,c_2\in(0,1)$: Let
$$c_1=\frac{\left(\alpha+\beta\right)^d}{\left(\alpha+\beta\right)^d+\left(\beta+\gamma\right)^d}$$
and let $c_2$ be the unique real solution of
\[(\alpha+\beta)^{c_2}(\beta+\gamma)^{1-c_2}=1.\]
Recall that $a_{d,w}$ consists of a binomial part $$b_{d,w}=\frac{1}{d!}\binom{n}{w}(\alpha+\beta)^{dw}(\beta+\gamma)^{d(n-w)} $$ and an exponential part $$e_{d,w}=\exp\left(-\left(\alpha+\beta\right)^w\left(\beta+\gamma\right)^{n-w}\right)$$ with $a_{d,w}=b_{d,w}e_{d,w}.$ The choice of $c_1$ and $c_2$ allows us to analyze $a_{d,w}$ more systematically, due to the following two observations: When \linebreak[4] $w>(1+o(1))c_1n$, then the binomial term $b_{d,w}$ starts decreasing significantly, and similarly when $w>c_2n$, then the exponential term $e_{d,w}$ starts decreasing at a significant rate. We need different arguments depending on the relation between $c_1$ and $c_2,$ thus we consider three cases. \\ 

(i)  If $c_1<c_2$, by application of the Binomial Theorem and since $e_{d,w}\leq 1,$ we have that
\[\mathbb{E}(\nvd_d)\stackrel{\eqref{ExpDeg}}{\leq} (1+o(1))\frac{\left(\left(\alpha+\beta\right)^d+\left(\beta+\gamma\right)^d\right)^n}{d!}.\]
On the other hand, set $c=\frac{c_1+c_2}{2}\in(c_1,c_2)$ and observe that \begin{equation}
\label{adc}a_{d,cn}=o\left(\left(\left(\alpha+\beta\right)^d+\left(\beta+\gamma\right)^d\right)^n\right).
\end{equation} Moreover, we have for all $w\leq cn$ that $$(\alpha+\beta)^w(\beta+\gamma)^{n-w}=o(1),$$ and therefore the exponential term satisfies $e_{d,w}=1+o(1).$ Thus, by dropping some summands, it follows that
\begin{align*}
\mathbb{E}(\nvd_d)&\stackrel{\eqref{ExpDeg}}{\geq} (1+o(1))\sum_{w=0}^{cn} \binom{n}{w} \frac{(\alpha+\beta)^{dw}(\beta+\gamma)^{d(n-w)}}{d!}\\
&\geq (1+o(1))\frac{\left((\alpha+\beta)^d+(\beta+\gamma)^d\right)^n}{d!} 
\end{align*}
by application of Lemma~\ref{binomappr} and Estimate~\eqref{adc}. Hence Statement~\eqref{stmt1} is satisfied.\\

(ii) If $c_1=c_2$, we split the sum into three parts
\begin{equation}\label{sumCase6ii}
\mathbb{E}(\nvd_d)\stackrel{\eqref{ExpDeg}}{=}(1+o(1))\left(\sum_{w=0}^{c_1n-\log{n}}a_{d,w}+\sum_{w=c_1n-\log{n}}^{c_1n+\log{n}}a_{d,w}+\sum_{w=c_1n+\log{n}}^{n}a_{d,w}\right)+o(1).\end{equation}
Now note that we can express the binomial part $b_{d,w}$ of the summand $a_{d,w}$ with the help of a binomially distributed random variable \begin{equation}\label{binCase6}
b_{d,w}=\frac{1}{d!}\left(\left(\alpha+\beta\right)^d+\left(\beta+\gamma\right)^d\right)^n\mathbb{P}\left(\mathrm{Bin}\left(n,c_1\right)=w\right).
\end{equation} Using the normal approximation of the binomial distribution we will be able to estimate the three parts of the sum in Equation \eqref{sumCase6ii}. First of all, since $e_{d,w}\leq 1,$ we can estimate the second term
\begin{align*}\sum_{w=c_1n-\log{n}}^{c_1n+\log{n}} \hspace{-0.3cm}a_{d,w} 
&\hspace{-0.1cm}\stackrel{\eqref{binCase6}}{=}\hspace{-0.1cm}O\hspace{-0.1cm}\left[\frac{1}{d!}\left(\left(\alpha+\beta\right)^d+\left(\beta+\gamma\right)^d\right)^n\mathbb{P}\left(\left|\mathrm{N}\left(0,1\right)\right|\leq \frac{\log n}{c_1(1-c_1)n}\right)\right]\\
&=o\left(\left(\left(\alpha+\beta\right)^d+\left(\beta+\gamma\right)^d\right)^n\right).\end{align*}
Similarly, since $e_{d,w}=o(1)$ for the summands of the last term, we also have
\begin{align*}
\sum_{w=c_1n+\log{n}}^{n}\hspace{-0.2cm} a_{d,w}&\stackrel{\eqref{binCase6}}{=}o\hspace{-0.06cm}\left[\frac{1}{d!}\left(\left(\alpha+\beta\right)^d+\left(\beta+\gamma\right)^d\right)^n\mathbb{P}\left(\mathrm{N}\left(0,1\right)\geq \frac{\log n}{c_1(1-c_1)n}\right)\right]\\
&=o\left(\left(\left(\alpha+\beta\right)^d+\left(\beta+\gamma\right)^d\right)^n\right).
\end{align*}
Finally, we have $e_{d,w}=1+o(1)$ for the summands in the  first term, hence we get
\begin{align*}
\sum_{w=0}^{c_1n-\log{n}}a_{d,w}&\stackrel{\eqref{binCase6}}{=}(1+o(1))\frac{1}{d!}\left(\left(\alpha+\beta\right)^d+\left(\beta+\gamma\right)^d\right)^n\mathbb{P}\left(\mathrm{N}\left(0,1\right)\leq 0\right)\\
&=(1+o(1))\frac{1}{2d!}\left(\left(\alpha+\beta\right)^{d}+\left(\beta+\gamma\right)^{d}\right)^n,
\end{align*}
and thus Statement \eqref{stmt1} holds.\\

(iii) If $c_1>c_2$, the sum can be split into two parts,
\begin{equation}\label{sumCase6iii}
\mathbb{E}(\nvd_d)\stackrel{\eqref{ExpDeg}}{=}(1+o(1))\left(\sum_{w=0}^{c_2n+\iota}a_{d,w}+\sum_{w=c_2n+\iota+1}^{n}a_{d,w}\right)+o(1),
\end{equation} 
where $\iota$ is some constant which will be determined later. Our goal is to show $\mathbb{E}\left(\nvd_d\right)=\Theta\left(a_{d,c_2n+\iota}\right).$ 
This implies $\mathbb{E}\left(\nvd_d\right)=o\left(2^n\right),$ i.e.\ Statement \eqref{stmt2} holds, since by the definition of $c_2$ and Stirling's approximation for binomial coefficients we know that 
$$a_{d,c_2n+\iota}\leq b_{d,c_2n+\iota}=\frac{1}{d!}\binom{n}{c_2n+\iota}=o(2^n). $$ 
We begin the analysis with the second sum of \eqref{sumCase6iii}. Note that due to our choice of $c_2$, we have that for any $m\geq 0$
\begin{align*}
a_{d,c_2n+m}&=(\alpha+\beta)^{dc_2n}(\beta+\gamma)^{dc_2n}\left(\frac{\alpha+\beta}{\beta+\gamma}\right)^{dm}\\
&\hspace{4cm}\cdot\exp\left(-(\alpha+\beta)^{c_2n}(\beta+\gamma)^{c_2n}\left(\frac{\alpha+\beta}{\beta+\gamma}\right)^{m}\right)\\
&=\left(\frac{\alpha+\beta}{\beta+\gamma}\right)^{dm}\exp\left(-\left(\frac{\alpha+\beta}{\beta+\gamma}\right)^{m}\right).
\end{align*}
If we consider the quotient of two successive summands we obtain
$$\frac{a_{d,c_2n+m+1}}{a_{d,c_2n+m}}=\left(\frac{\alpha+\beta}{\beta+\gamma}\right)^d\exp\left(\left(1-\frac{\alpha+\beta}{\beta+\gamma}\right)\left(\frac{\alpha+\beta}{\beta+\gamma}\right)^m\right)$$
and since $(\alpha+\beta)/(\beta+\gamma)>1$ the sequence of these quotients is monotone decreasing in $m$. Now define $\iota$ as the smallest positive integer such that
$$\zeta_{d,\iota}:=\left(\frac{\alpha+\beta}{\beta+\gamma}\right)^d\exp\left(\left(1-\frac{\alpha+\beta}{\beta+\gamma}\right)\left(\frac{\alpha+\beta}{\beta+\gamma}\right)^\iota\right)<1$$
and note that the value of $\iota$ does not depend on $n$.
Therefore
\begin{align*}
\sum_{w=c_2n+\iota+1}^{n}a_{d,w}&\leq a_{d,c_2n+\iota}\sum_{i=1}^{n-c_2n-\iota} \left(\zeta_{d,\iota}\right)^i,
\end{align*} and since the sum can be bounded from above by a convergent geometric series this is $O(a_{d,c_2n+\iota}).$

On the other hand for the first summand in \eqref{sumCase6iii} we have that
$$\sum_{w=0}^{c_2n+\iota}a_{d,w}\leq \sum_{w=0}^{c_2n+\iota}b_{d,w}.$$
Since $c_2n+\iota=(1+o(1))c_2n$ and $c_2<c_1$ thus Lemma \ref{binomappr} implies that
$$\sum_{w=0}^{c_2n+\iota}b_{d,w} =O(b_{d,c_2n+\iota}).$$
Also $\iota$ does not depend on $n$ thus $e_{d,c_2n+\iota}$ is a constant and therefore $b_{d,c_2n+\iota}=O(a_{d,c_2n+\iota})$ completing Case $3$(iii).\\

\bf Case 4:\hspace{.2cm} $\beta+\gamma\leq\alpha+\beta<1.$\normalfont \\
Note in this case that in this case $e_{d,w}=(1+o(1))$ uniformly and thus
\begin{align*}
\mathbb{E}(\nvd_d)&\stackrel{\eqref{ExpDeg}}{=}\frac{(1+o(1))}{d!}\sum_{w=0}^n\binom{n}{w}(\alpha+\beta)^{dw}(\beta+\gamma)^{d(n-w)}+o(1)\\
&=(1+o(1))\frac{\left((\alpha+\beta)^{d}+(\beta+\gamma)^{d}\right)^n}{d!}+o(1).
\end{align*}
Hence Statement \eqref{stmt1} holds.\\

\bf Case 5:\hspace{.2cm} $1<\beta+\gamma\leq \alpha+\beta.$\normalfont \\
Observe that we get
\[\mathbb{E}\left(\nvd_d\right)\stackrel{\eqref{ExpDeg}}{\leq}(1+o(1))\exp\left(-(\beta+\gamma)^n\right)\left((\alpha+\beta)^d+(\beta+\gamma)^d\right)^n+o(1)=o(1),\]
and Statement \eqref{stmt2} holds.\\

\bf Case 6:\hspace{.2cm} $\beta+\gamma=\alpha+\beta=1.$\normalfont\\
We get
\[\mathbb{E}\left(\nvd_d\right)\stackrel{\eqref{ExpDeg}}{=}(1+o(1))\frac{2^n}{\mathrm{e}d!},\]
thus Statement \eqref{stmt1} holds  and thereby we complete the proof of Lemma~\ref{lem:expdeg}.
\end{proof}

With the help of these preliminary results we can now prove Theorem~\ref{main}.

\begin{proof}[Proof of Theorem \ref{main}]
In order for a graph to have a power law degree distribution it is necessary that the number of vertices of degree $d$ is approximately $cd^{-\beta}|
V|$; in particular, for every finite $d$ the expected number of vertices with degree $d$ has to be a linear fraction of all vertices. In Lemma~\ref{lem:expdeg} we have shown that either \begin{equation*}
\mathbb{E}(\nvd_d)=\Theta(((\alpha+\beta)^d+(\beta+\gamma)^d)^n),
\end{equation*} 
or 
\begin{equation*}
\mathbb{E}\left(\nvd_d\right)=o\left(2^n\right).
\end{equation*}
 If $\mathbb{E}\left(\nvd_d\right)=o\left(2^n\right)$, then Markov's inequality implies that the stochastic Kronecker graph a.a.s.\ does not follow a power law degree distribution. Clearly, the only parameter choice which can satisfy $\mathbb{E}(\nvd_d)=\Theta(2^n)$ for every finite $d$ is when $(\alpha+\beta)^d+(\beta+\gamma)^d=2$. However, this can hold only if $\alpha+\beta=\beta+\gamma=1$. A closer examination of this case gives us that 
\begin{equation}\label{EZd}
\mathbb{E}\left(\nvd_d\right)=(1+o(1))\frac{2^n}{\mathrm{e}d!}=(1+o(1))\frac{1}{\mathrm{e}d!}|V|,
\end{equation}
which indicates that the number of edges follows a Poisson distribution with parameter 1 and not a power law degree distribution. In fact, we will show that a.a.s.\ $$\nvd_d=(1+o(1))\mathbb{E}(\nvd_d).$$ Recall that for a vertex $u$ we denote by $N(u)$ the set of its neighbors and by $d(u)=\left|N\left(u\right)\right|$ its degree. Conditioning on the edge $\{u,v\}$ being present or not, we can estimate
\[\mathbb{P}\left(d\left(u\right)=d\right)-p_{u,v} \leq \mathbb{P}\left(\left|N\left(u\right)\backslash v\right|=d\right) \leq \mathbb{P}\left(d\left(u\right)=d\right)+p_{u,v},\]
and thus we have 
\begin{equation}\label{Prob}
\mathbb{P}\left(\left|N\left(u\right)\backslash v\right|=d\right)=(1+o(1))\mathbb{P}\left(d\left(u\right)=d\right)=(1+o(1))e^{-1}/d!\, .
\end{equation} 
Since the events $\left|N\left(u\right)\backslash v\right|=d$ and $\left|N\left(v\right)\backslash u\right|=d$ are independent, the second moment of $\nvd_d$ satisfies
\begin{align*}
 \mathbb{E}\left(\nvd_d^2\right)&=\sum_{u\in V} \mathbb{P}\left(d\left(u\right)=d\right) + \sum_{\substack{u,v\in V\\u\neq v}} \mathbb{P}\left(d\left(u\right)=d\left(v\right)=d\right) \\
&\leq \mathbb{E}\left(\nvd_d\right)+ \sum_{\substack{u,v\in V\\u\neq v}}\Big[ \mathbb{P}\left(\left|N\left(u\right)\backslash v\right|=d\right)\mathbb{P}\left(\left|N\left(v\right)\backslash u\right|=d\right)+p_{u,v}\Big]\\
&\stackrel{\eqref{Prob}}{=} (1+o(1))4^n \left(\frac{e^{-1}}{d!}\right)^2\stackrel{\eqref{EZd}}{=}(1+o(1))\left(\mathbb{E}\left(\nvd_d\right)\right)^2.
\end{align*}
The statement follows by applying the second moment method, i.e.\ by Chebyshev's inequality.
\end{proof}

\section{Small subgraphs: proofs of Lemma~\ref{subgraphs} and Theorem~\ref{cycleconc}}\label{ch:subgraphs}

Let $G$ be a fixed simple graph and label the vertices of $G$ with $1,\ldots,\left|V(G)\right|$.  Denote by $X_G$ the number of labeled copies of $G$ in the stochastic Kronecker graph $K(n,P)$.
Let $L_G$ be the set of functions $g:V(G)\rightarrow \mathbb{Z}_2$.
Define the base value of a graph to be
\[B_G=\sum_{g\in L_G}\prod_{\{u,v\}\in E(G)}P_{g(u),g(v)}\]
and for any fixed function $g\in L_G$ let its contribution to the base value of $G$ be
\[b_G(g)=\prod_{\{u,v\}\in E(G)}P_{g(u),g(v)}.\]

We now establish the expectation of the number of subgraphs in $K(n,P)$.

\begin{proof}[Proof of Lemma \ref{subgraphs}]
First we express $X_G$ as a sum of indicator random variables, one for each \emph{injective} function of the vertex set $V(G)$ into the vertex set $\mathbb{Z}_2^n$ of the stochastic Kronecker graph $K(n,P).$ Then we obtain an upper bound $U_G$ for the expectation of the number of copies of $G$ in $K(n,P)$ by summing over \emph{all} mappings from $V(G)$ into $\mathbb{Z}_2^n$.
\begin{align*}
\mathbb{E}(X_G)&=\sum_{\substack{\varphi:V(G)\to\mathbb{Z}_2^n\\\varphi \text{ inj.}}}\prod_{\{u,v\}\in E(G)}\hspace{-0.2cm} p_{\varphi(u),\varphi(v)}
\leq\sum_{\varphi:V(G)\to\mathbb{Z}_2^n}\prod_{\{u,v\}\in E(G)}\hspace{-0.2cm}p_{\varphi(u),\varphi(v)}=U_G.
\end{align*} 
Using the digit-wise representation of the probabilities $p_{\varphi(i),\varphi(j)}$ we get
\begin{align*}
U_G&=\sum_{\varphi:V(G)\to\mathbb{Z}_2^n}\prod_{\{u,v\}\in E(G)}\prod_{k=1}^n P_{\varphi(u)_k,\varphi(v)_k}.
\end{align*} 
Furthermore, note that, for any $k\in[n]$ and any function $\varphi:V(G)\to\mathbb{Z}_2^n$, restricting $\varphi$ to the $k$-th digit defines a function $g_k\in L_G$ with $g_k(w)=\varphi(w)_k$ for all vertices $w\in V(G).$ Therefore by changing the order of summation we can express the upper bound $U_G$ in terms of the base value $B_G$ $$U_G=\prod_{k=1}^{n}\sum_{g_k\in L_G}\prod_{\{u,v\}\in E(G)}\hspace{-0.2cm}P_{g_k(u),g_k(v)}
=\left(\sum_{g\in L_G}\prod_{\{u,v\}\in E(G)}\hspace{-0.2cm}P_{g(u),g(v)}\right)^n\hspace{-0.2cm}=\left(B_G\right)^n.$$

It remains to show that this upper bound is asymptotically tight. Therefore let us consider the error-term
\begin{align*}
\left(B_G\right)^n-\mathbb{E}(X_G)&=\hspace{-0.2cm}\sum_{\varphi:V(G)\to\mathbb{Z}_2^n}\prod_{\{u,v\}\in E(G)}\hspace{-0.2cm}p_{\varphi(u),\varphi(v)}-\hspace{-0.2cm}\sum_{\substack{\varphi:V(G)\to\mathbb{Z}_2^n\\\varphi \text{ inj.}}}\prod_{\{u,v\}\in E(G)}\hspace{-0.2cm}p_{\varphi(u),\varphi(v)}
\end{align*}
and note that the proof is complete if we show that it is $o\left(\left(B_G\right)^n\right)$.

Fix a pair of vertices $\{w_1,w_2\}\subset V(G)$ with $w_1\neq w_2$ and note that by a similar argument we have
\begin{align*}
\sum_{\substack{\varphi:V(G)\to\mathbb{Z}_2^n\\\varphi(w_1)=\varphi(w_2)}}\prod_{\{u,v\}\in E(G)}p_{\varphi(u),\varphi(v)}=\left(\sum_{\substack{g\in L_G\\g(w_1)=g(w_2)}}\prod_{\{u,v\}\in E(G)}P_{g(u),g(v)}\right)^n.
\end{align*}
Since all entries of the probability matrix $P$ are positive, in other words $\alpha,\beta,\gamma>0$, every $g\in L_{G}$ has a positive contribution $b_G(g)>0$ to the base value of $G$ and thus we have for every factor
\[\vartheta:=\sum_{\substack{g\in L_G\\g(w_1)=g(w_2)}}\prod_{\{u,v\}\in E(G)}P_{g(u),g(v)}<\sum_{g\in L_G}\prod_{\{u,v\}\in E(G)}P_{g(u),g(v)}=B_G.\]
Hence, as neither $\vartheta$ nor $B_G$ depend on $n$ we have that $\vartheta^n=o((B_G)^n)$
completing the proof.
\end{proof}

In order to show concentration for $X_G$ we will apply the second moment method. This means that we want to compare the variance of $X_G$ with its expectation. By expressing $X_G$ as a sum of indicator random variables we shall calculate the covariance for any two of the indicator random variables that are not independent. Therefore we have to consider any fixed graph $F$ formed by two edge-overlapping copies of $G$ and determine the expectation of $X_F,$ i.e. the expected number of copies of $F$ contained in $K(n,P).$ 

To this end, we characterize the set of graphs consisting of two overlapping copies of a fixed graph with the help of graph homomorphisms. 
But first note that, given two graphs $H_1$ and $H_2,$ any graph homomorphism $\psi:V\left(H_1\right)\rightarrow V\left(H_2\right)$  
canonically extends to 
a function $\widehat{\psi}:E\left(H_1\right)\rightarrow E\left(H_2\right)$, $\widehat{\psi}(\{u,v\})=\{\psi(u),\psi(v)\}$, which we call \emph{edge function}. 

\begin{definition}\label{def:graphunion}
Define $\mathcal{F}_G$ as the set of graphs $F$ such that there exist injective homomorphisms $f_1,f_2:V(G)\rightarrow V(F)$ such that the following holds.
\begin{itemize}
\item $V(F)=f_1(V(G))\cup f_2(V(G))$;
\item $E(F)=\widehat{f_1}(E(G))\cup \widehat{f_2}(E(G))$;
\item $\widehat{f_1}(E(G))\neq \widehat{f_2}(E(G))$;
\item $\widehat{f_1}(E(G))\cap \widehat{f_2}(E(G))\neq \emptyset$.
\end{itemize}
\end{definition}

With this notation, the following concentration lemma is a direct application of Chebyshev's inequality, .

\begin{lemma}[\cite{MR1782847}]\label{conc}
Let $G$ be a fixed graph. If for every graph $F\in \mathcal{F}_G$ we have that $\mathbb{E}(X_F)=o((\mathbb{E}(X_G))^2)$, then a.a.s.\ $X_G=(1+o(1))\mathbb{E}(X_G)$. 
\end{lemma}

Therefore, we shall compare the base values of two graphs $H_1$ and $H_2$. The following Lemma will prove to be a useful tool for this.

\begin{lemma}\label{comp}
Let $H_1$ and $H_2$ be simple graphs. If there exists a surjective homomorphism $\phi:V(H_1)\rightarrow V(H_2)$ such that $\widehat{\phi}$ is injective, then $$B_{H_1}\geq B_{H_2}.$$
\end{lemma}

\begin{proof} 
Define a function $\Phi:L_{H_2}\rightarrow L_{H_1}$ by setting $\Phi(h_2)(v)=h_2(\phi(v))$ for every $v\in V(H_1)$ and $h_2\in L_{H_2}.$ 
Note that $\Phi$ is injective since $\phi$ is surjective providing a bijection \begin{equation}\label{bijPhi}
\Phi\left(L_{H_2}\right)\simeq L_{H_2},
\end{equation} and similarly, since $\widehat{\phi}$ is also injective, a bijection
\begin{equation}\label{bijphi}
\widehat{\phi}\left(E\left(H_1\right)\right)\simeq E\left(H_1\right).
\end{equation} 
Thus, we have
\begin{align*}
B_{H_2}&=\sum_{h_2\in L_{H_2}}\prod_{\{u,v\}\in E\left(H_2\right)}P_{h_2(u),h_2(v)}\\
&\leq \sum_{h_2\in L_{H_2}}\prod_{\{u,v\}\in \widehat{\phi}\left(E\left(H_1\right)\right)}P_{h_2(u),h_2(v)}\\
&\hspace{-2pt}\stackrel{\eqref{bijphi}}{=}\sum_{h_2\in L_{H_2}}\prod_{\{u,v\}\in E\left(H_1\right)}P_{h_2(\phi(u)),h_2(\phi(v))}\\
&=\sum_{h_2\in L_{H_2}}\prod_{\{u,v\}\in E\left(H_1\right)}P_{\Phi(h_2)(u),\Phi(h_2)(v)}\\
&\hspace{-2pt}\stackrel{\eqref{bijPhi}}{=}\sum_{h_1\in \Phi\left(L_{H_2}\right)}\prod_{\{u,v\}\in E\left(H_1\right)}P_{h_1(u),h_1(v)}\\
&\leq\sum_{h_1\in L_{H_1}}\prod_{\{u,v\}\in E\left(H_1\right)}P_{h_1(u),h_1(v)}=B_{H_1}. 
\end{align*}
\end{proof}

The standard application of Lemma \ref{comp} is the following: for any graph $H$ and any two vertices $u,v\in V(H),$ we create a graph $H'$ from $H$ by identifying $u$ and $v$, i.e.\ removing $v$ and all of the edges adjacent to $v$ and inserting an edge $\{u,w\}$ when $\{v,w\}\in E(H)$ and $w\neq v$. If this does not create any multiple edges, i.e.\ if $H'$ is a simple graph, then Lemma \ref{comp} implies that $B_{H'}\leq B_H$.

With the help of this we will now study the thresholds for the appearance of some classes of subgraphs and we start with stars.

\begin{theorem}\label{2ndMM}
Let $k\in\mathbb{N}.$ The threshold for the appearance of $K_{1,k}$ in the stochastic Kronecker graph is $(\alpha+\beta)^k+(\beta+\gamma)^k=1$. Additionally, if $(\alpha+\beta)^k+(\beta+\gamma)^k>1$, then a.a.s.\ $X_{K_{1,k}}=(1+o(1))\big((\alpha+\beta)^k+(\beta+\gamma)^k\big)^n$.
\end{theorem}

\begin{proof}
Assume that in a labeling $g\in L_{K_{1,k}}$ the central vertex is labeled with 1 and exactly $i$ of the remaining vertices are labeled 1. Then we have  $b_{K_{1,k}}(g)=\alpha^{i}\beta^{k-i}$. Note that there are $\binom{k}{i}$ ways to create such a labeling. Similarly, if the central vertex is labeled 0 and exactly $i$ of the remaining vertices is labeled 0, then $b_{K_{1,k}}(g)=\gamma^{i}\beta^{k-i}$ and there are $\binom{k}{i}$ ways to create such a labeling. Thus, we have
\[B_{K_{1,k}}=\sum_{i=0}^k\binom{k}{i}\alpha^i\beta^{k-i}+\sum_{i=0}^k\binom{k}{i}\gamma^i\beta^{k-i}=(\alpha+\beta)^k+(\beta+\gamma)^k,\]
and therefore, if $(\alpha+\beta)^k+(\beta+\gamma)^k<1,$ we get $$\mathbb{E}\left(X_{K_{1,k}}\right)=(1+o(1))\left(B_{K_{1,k}}\right)^n=o(1).$$ In particular a.a.s.\ there is no copy of $K_{1,k}$ contained in $K(n,P).$  

Now for the rest of the proof we assume $(\alpha+\beta)^k+(\beta+\gamma)^k>1.$ Note that, by Lemma~\ref{conc}, it is sufficient to show that for every $F\in \mathcal{F}_{K_{1,k}}$ we have  $B_F<\left(B_{K_{1,k}}\right)^2.$ Fix a graph $F\in \mathcal{F}_{K_{1,k}}$ and let $f_1$ and $f_2$ be as in Definition \ref{def:graphunion}. Let $u$ denote the central vertex of the star $K_{1,k}$. There are two options: either the central vertices of the copies match, i.e.\ $f_1(u)=f_2(u)$, or they do not. If $f_1(u)=f_2(u)$, then we have that $F$ is a star namely $F=K_{1,k+\ell}$ for some $0< \ell <k$. 

On the other hand, if $f_1(u)\neq f_2(u)$, then the two copies of $K_{1,k}$ found in $F$ overlap in exactly one edge, i.e.\ $|\hat{f}_1(E(K_{1,k}))\cap \hat{f}_2(E(K_{1,k}))|=1$. Now let $F^*\in \mathcal{F}_{K_{1,k}}$ be the graph such that $f_1(u)\neq f_2(u)$ and the two graphs overlap only in the two end-vertices of the edge, i.e.\ $f_1(V(K_{1,k}))\cap f_2(V(K_{1,k}))=\{f_1(u)\cup f_2(u)\}$. 
Lemma \ref{comp} implies that for every graph $F'\in \mathcal{F}_{K_{1,k}}$ such that the central vertices of the copies of $K_{1,k}$ do not match, we have  $B_{F'}\leq B_{F^*}$ (see Figure \ref{fig:stars}).

\begin{figure}[h]
\begin{center}
\begin{tikzpicture}[scale=0.9]
\node (X0) at (14-1,0) [circle,fill=black, inner sep=1pt,label={below:1},label=left:{$f_1(u)$}] {};
\node (X1) at (14,1) [circle,fill=black, inner sep=1pt,label=above:{$f_1(v_1)=f_2(v_2)$},label=below:?] {};
\node (X2) at (14-2,1) [circle,fill=black, inner sep=1pt,label={left:1}] {};
\node (X3) at (14-2,-1) [circle,fill=black, inner sep=1pt,label={left:1}] {};
\node (X4) at (14-0.25,-1) [circle,fill=black, inner sep=1pt,label={left:1}] {};
\node (Y0) at (15,0) [circle,fill=black, inner sep=1pt,label={below:1},label=right:{$f_2(u)$}] {};
\node (Y2) at (16,1) [circle,fill=black, inner sep=1pt,label={right:1}] {};
\node (Y3) at (16,-1) [circle,fill=black, inner sep=1pt,label={right:1}] {};
\node (Y4) at (14.25,-1) [circle,fill=black, inner sep=1pt,label={right:1}] {};
\node at (14,-1.75) {$F'$};
{
\draw[dotted] (X0)--(X1);
\draw[dotted] (X0)--(X2);
\draw[dotted] (X0)--(X3);
\draw[dotted] (X0)--(X4);
}
{
\draw[dashed] (Y0)--(X1);
\draw[dashed] (Y0)--(Y2);
\draw[dashed] (Y0)--(Y3);
\draw[dashed] (Y0)--(Y4);
}
{\draw (X0)--(Y0);}
\node (X10) at (5,0) [circle,fill=black, inner sep=1pt,label={below:1},label=left:{$f_1(u)$}] {};
\node (X11) at (5.5,1) [circle,fill=black, inner sep=1pt,label=above:{$f_1(v_1)$},label=below:0] {};
\node (X12) at (4,1) [circle,fill=black, inner sep=1pt,label={left:1}] {};
\node (X13) at (4,-1) [circle,fill=black, inner sep=1pt,label={left:1}] {};
\node (X14) at (6,-1) [circle,fill=black, inner sep=1pt,label={left:1}] {};
\node (Y10) at (7.5,0) [circle,fill=black, inner sep=1pt,label={below:1},label=right:{$f_2(u)$}] {};
\node (Y11) at (7,1) [circle,fill=black, inner sep=1pt,label=above:$f_2(v_2)$,label=below:1] {};
\node (Y12) at (8.5,1) [circle,fill=black, inner sep=1pt,label={right:1}] {};
\node (Y13) at (8.5,-1) [circle,fill=black, inner sep=1pt,label={right:1}] {};
\node (Y14) at (6.5,-1) [circle,fill=black, inner sep=1pt,label={right:1}] {};
\node at (6.25,-1.75) {$F^{*}$};
{
\draw[dotted](X10)--(X11);
\draw[dotted] (X10)--(X12);
\draw[dotted] (X10)--(X13);
\draw[dotted] (X10)--(X14);
}
{
\draw[dashed] (Y10)--(Y11);
\draw[dashed] (Y10)--(Y12);
\draw[dashed] (Y10)--(Y13);
\draw[dashed] (Y10)--(Y14);
}
{\draw (X10)--(Y10);}
\end{tikzpicture}
\caption{\label{fig:stars} {\it Edge overlaping copies of two $K_{1,5}$. Solid lines indicate overlaping edges, dashed and dotted lines indicate non-overlaping edges of the individual copies. A particular vertex labeling (numbers) is given for $F^*$ which is not transferable to $F'$ in the sense that it cannot be decided what the label of the identified vertex ($f_1(v_1)=f_2(v_2)$) should be, as the labels for $f_1(v_1)$ and $f_2(v_2)$ differ.}}
\end{center}
\end{figure}
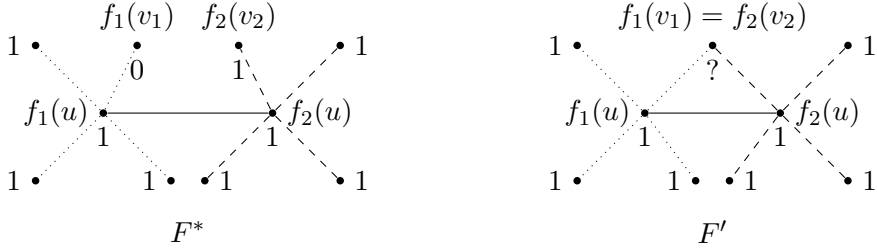
 Calculating the base value $B_{F^*}$ explicitly yields
\begin{align*}
B_{F^*}&=\alpha(\alpha+\beta)^{2k-2}+2\beta(\alpha+\beta)^{k-1}(\beta+\gamma)^{k-1}+\gamma(\beta+\gamma)^{2(k-1)}\\
&\leq (\alpha+\beta)^{2k-1}+(\beta+\gamma)^{2k-1}=B_{K_{1,2k-1}}.
\end{align*}
Thus, we only have to show that $B_{K_{1,k+\ell}}<\left(B_{K_{1,k}}\right)^2$ for $0<\ell<k$. If $\alpha+\beta\leq 1,$ we get, using $(\alpha+\beta)^k+(\beta+\gamma)^k>1,$ that
\begin{align*}
\left(B_{K_{1,k}}\right)^2=((\alpha+\beta)^k+(\beta+\gamma)^k)^2 &> (\alpha+\beta)^k+(\beta+\gamma)^k\\
&\geq (\alpha+\beta)^{\ell}\big((\alpha+\beta)^{k}+(\beta+\gamma)^{k}\big)\\
&\geq (\alpha+\beta)^{k+\ell}+(\beta+\gamma)^{k+\ell}=B_{K_{1,k+\ell}}\,.
\end{align*}
On the other hand, if $(\alpha+\beta)>1$, then we have
\begin{align*}
\left(B_{K_{1,k}}\right)^2=((\alpha+\beta)^k+(\beta+\gamma)^k)^2&> (\alpha+\beta)^{2k}+(\alpha+\beta)^k(\beta+\gamma)^k\\
&\geq (\alpha+\beta)^{2k}+(\alpha+\beta)^{k-\ell}(\beta+\gamma)^{k+\ell}\\
&=(\alpha+\beta)^{k-\ell}((\alpha+\beta)^{k+\ell}+(\beta+\gamma)^{k+\ell})\\
&\geq (\alpha+\beta)^{k+\ell}+(\beta+\gamma)^{k+\ell}=B_{K_{1,k+\ell}}\,.\,  \qedhere
\end{align*}
\end{proof}

Next we will show similar concentration results for trees and cycles, but only in the case when $\alpha=\gamma$. The arguments simplify if we examine edge labelings of a simple graph $G$, i.e.\ functions $\widehat{g}\in \widehat{L}_G:=\left\{\widehat{g}:E(G)\to \{0,1\}\right\},$ instead of vertex labelings of $G.$ For this we define a function $\Psi: L_G\to \widehat{L}_G$ by setting $\Psi(g)(\{u,v\})=\left|g(u)-g(v)\right|$ for all edges $\{u,v\}\in E(G)$ and vertex labelings $g\in L_G.$ Note that, if $G$ is connected, then there exist exactly two labelings $g_1,g_2\in L_G$ such that $\Psi(g_1)=\Psi(g_2)$. However, depending on $G$ there might exist edge labelings, for which no equivalent vertex labelings exist, e.g.\ labeling every edge of an odd cycle with 1. 
We call an edge labeling $\widehat{g}$ valid, if $\Psi^{-1}(\widehat{g})\neq \emptyset.$ The set of all valid edge labelings of $G$ is denoted by $\Psi\left(L_G\right)\subset\widehat{L}_G.$ We will once again be interested in the base value of $G$ but now will calculate by summing over all valid edge labelings. 

Now let $\alpha=\gamma.$ We define $b_G\left(\widehat{g}\right):=\alpha^{|\hat{g}^{-1}(0)|}\beta^{|\hat{g}^{-1}(1)|}$ to be the contribution to the base value of $G$ of a fixed edge labeling $\widehat{g}\in \Psi\left(L_G\right)$ and thus we get for every connected graph $G$ that its base value satisfies
\begin{equation}\label{eq:expectation}
B_G=\sum_{g\in L_G}b_G(g)=2\sum_{\widehat{g}\in \Psi\left(L_G\right)}\alpha^{|\widehat{g}^{-1}(0)|}\beta^{|\widehat{g}^{-1}(1)|}\,.
\end{equation}

With the help of this observation we now establish concentration results for trees and cycles.

\begin{theorem}\label{tree}
Let $T$ be a tree and assume that $\alpha=\gamma$. The threshold for the appearance of $T$ in the stochastic Kronecker graph is $2(\alpha+\beta)^{e(T)}=1$. Additionally, if $2(\alpha+\beta)^{e(T)}>1$, then a.a.s.\ $X_{T}=(1+o(1))(2(\alpha+\beta)^{e(T)})^n$.
\end{theorem}

\begin{proof}
Note that for a tree $T,$ every labeling is valid and Equation \eqref{eq:expectation} implies that 
\[\mathbb{E}(X_T)=(1+o(1))(2(\alpha+\beta)^{e(T)})^n,\] and the first statement follows by Markov's inequality.

Now assume $B_T>1$. We will again use the second moment method in form of Lemma~\ref{2ndMM}. Let $F'\in \mathcal{F}_T$ and let $F$ be a tree such that $e(F)=e(F')$. Since every labeling of the edges of $F$ is valid but, on the other hand, some labelings of $F'$ may not be valid, we have $B_{F}\geq B_{F'}$. Therefore, it is enough to show that for every tree $F$ such that $e(T)<e(F)<2e(T)$ we have that $B_F<\left(B_T\right)^2$. 
If $\alpha+\beta\leq 1$, we have
\begin{align*}
\left(B_T\right)^2=(2(\alpha+\beta)^{e(T)})^2> 2(\alpha+\beta)^{e(T)} \geq 2(\alpha+\beta)^{e(F)}=B_F.
\end{align*}
On the other hand, if $\alpha+\beta> 1$, we have
\begin{align*}
\left(B_T\right)^2=(2(\alpha+\beta)^{e(T)})^2= 4(\alpha+\beta)^{2e(T)} > 2(\alpha+\beta)^{e(F)}=B_F,
\end{align*}
completing the proof.
\end{proof}

We conclude this section with the proof of Theorem~\ref{cycleconc} concerning the threshold for the appearance of cycles in $K(n,P).$

\begin{proof}[Proof of Theorem \ref{cycleconc}]
Unlike in the case of trees, the edge labelings of $C_k$ are not necessarily valid. However, note that the only reason for a labeling not be valid is because it contains an odd number of $1$'s.  Therefore, Equation \eqref{eq:expectation} implies that
\[B_{C_k}=2\sum_{i=0}^{\lfloor k/2 \rfloor}\binom{k}{2i}\beta^{2i}\alpha^{k-2i}=(\alpha+\beta)^k+(\alpha-\beta)^k,\] and thereby the first statement of Theorem~\ref{cycleconc} holds due to Markov's inequality.

So now assume $B_{C_k}>1$ and proceed analogously to the proof of Theorem~\ref{tree}. First, for every $F'\in \mathcal{F}_{C_k}$, we will determine a graph $F$ such that $e(F)=e(F')$ and $B_F\geq B_{F'}$. 

Let $F_{\ell}$ be the graph created from two overlapping $k$-cycles such that the two cycles overlap in exactly $\ell$ consecutive edges.
Equivalently, $F_{\ell}$ is a pair of vertices $u,v$ which are connected by 3 vertex disjoint paths, where one consists of $\ell$ edges and the other two consist of $k-\ell$ edges. Note that a labeling of $F_{\ell}$ is valid if and only if the parity of the number of ones assigned to each of these paths is equal.
Therefore, by Equation \eqref{eq:expectation} we have 
\begin{align*}
B_{F_{\ell}}&=2\left(\sum_{i=0}^{\lfloor (k-\ell)/2\rfloor} \binom{k-\ell}{2i}\alpha^{k-\ell-2i}\beta^{2i} \right)^2\left(\sum_{i=0}^{\lfloor \ell/2\rfloor} \binom{\ell}{2i}\alpha^{\ell-2i}\beta^{2i}\right)\\
&\hspace{-.45cm}+2\left(\sum_{i=1}^{\lceil (k-\ell)/2\rceil} \binom{k-\ell}{2i-1}\alpha^{k-\ell-(2i-1)}\beta^{2i-1} \right)^2\left(\sum_{i=1}^{\lceil \ell/2\rceil} \binom{\ell}{2i-1}\alpha^{\ell-2i+1}\beta^{2i-1}\right)\\
&=2\left(\frac{(\alpha+\beta)^{k-\ell}+(\alpha-\beta)^{k-\ell}}{2}\right)^2\left(\frac{(\alpha+\beta)^{\ell}+(\alpha-\beta)^{\ell}}{2}\right)\\
&\hspace{+.8cm}+2\left(\frac{(\alpha+\beta)^{k-\ell}-(\alpha-\beta)^{k-\ell}}{2}\right)^2\left(\frac{(\alpha+\beta)^{\ell}-(\alpha-\beta)^{\ell}}{2}\right)
\end{align*}
where the second equality holds by several applications of the Binomial Theorem. Furthermore, this simplifies to 
\begin{align}
\label{baseFl1} B_{F_\ell}\hspace{-0.1cm}&=\hspace{-0.1cm}\frac{1}{2}\hspace{-0.1cm}\left((\alpha+\beta)^{2k-\ell}+(\alpha+\beta)^{\ell}(\alpha-\beta)^{2k-2\ell}+2(\alpha+\beta)^{k-\ell}(\alpha-\beta)^k\right)\\
&=\hspace{-0.1cm}\frac{1}{2}(\alpha+\beta)^{2k-\ell}\left(1+\left(\frac{\alpha-\beta}{\alpha+\beta}\right)^{2k-2\ell}+2\left(\frac{\alpha-\beta}{\alpha+\beta}\right)^k\right).\label{baseFl2}
\end{align}

Now let $F'\in \mathcal{F}_{C_k}$ such that $e(F')=2k-\ell$ and let $f_1',$ $f_2'$ be the corresponding functions as in Definition \ref{def:graphunion}. The edges of $F'$ can be partitioned into three sets: 
\begin{align*}
E_0&:=\widehat{f_1'}(E(C_k))\cap \widehat{f_2'}(E(C_k))\text{, }\\
 E_1&:=\widehat{f_1'}(E(C_k))\backslash \widehat{f_2'}(E(C_k)),\\
 E_2&:=\widehat{f_2'}(E(C_k))\backslash \widehat{f_1'}(E(C_k)).
\end{align*}

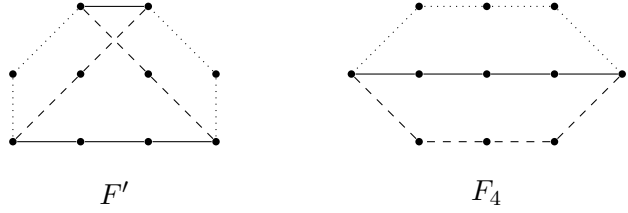
\begin{figure}[h]
\begin{center}
\begin{tikzpicture}[scale=0.9]

\node (Y0) at (-7,1) [circle,fill=black, inner sep=1pt,label={}] {};
\node (Y1) at (-8,1) [circle,fill=black, inner sep=1pt,label={}] {};
\node (Y2) at (-9,0) [circle,fill=black, inner sep=1pt,label={}] {};
\node (Y3) at (-8,0) [circle,fill=black, inner sep=1pt,label={}] {};
\node (Y4) at (-7,0) [circle,fill=black, inner sep=1pt,label={}] {};
\node (Y5) at (-6,0) [circle,fill=black, inner sep=1pt,label={}] {};
\node (Y6) at (-6,-1) [circle,fill=black, inner sep=1pt,label={}] {};
\node (Y7) at (-7,-1) [circle,fill=black, inner sep=1pt,label={}] {};
\node (Y8) at (-8,-1) [circle,fill=black, inner sep=1pt,label={}] {};
\node (Y9) at (-9,-1) [circle,fill=black, inner sep=1pt,label={}] {};
\node at (-7.5,-1.75){$F'$};
{
\draw[dotted] (Y5)--(Y6);
\draw[dotted] (Y5)--(Y0);
\draw[dotted] (Y1)--(Y2);
\draw[dotted] (Y2)--(Y9);
}

{
\draw[dashed] (Y0)--(Y3);
\draw[dashed] (Y3)--(Y9);
\draw[dashed] (Y1)--(Y4);
\draw[dashed] (Y4)--(Y6);
}

{
\draw (Y0)--(Y1);
\draw (Y7)--(Y8);
\draw (Y8)--(Y9);
\draw (Y6)--(Y7);
}

\node (X0) at (-1,1) [circle,fill=black, inner sep=1pt,label={}] {};
\node (X1) at (-2,1) [circle,fill=black, inner sep=1pt,label={}] {};
\node (X2) at (-3,1) [circle,fill=black, inner sep=1pt,label={}] {};
\node (X3) at (0,0) [circle,fill=black, inner sep=1pt,label={}] {};
\node (X4) at (-1,0) [circle,fill=black, inner sep=1pt,label={}] {};
\node (X5) at (-2,0) [circle,fill=black, inner sep=1pt,label={}] {};
\node (X6) at (-3,0) [circle,fill=black, inner sep=1pt,label={}] {};
\node (X7) at (-4,0) [circle,fill=black, inner sep=1pt,label={}] {};
\node (X8) at (-1,-1) [circle,fill=black, inner sep=1pt,label={}] {};
\node (X9) at (-2,-1) [circle,fill=black, inner sep=1pt,label={}] {};
\node (X10) at (-3,-1) [circle,fill=black, inner sep=1pt,label={}] {};
\node at (-2,-1.75){$F_4$};
{
\draw[dotted] (X0)--(X1);
\draw[dotted] (X1)--(X2);
\draw[dotted] (X2)--(X7);
\draw[dotted] (X0)--(X3);
}

{
\draw[dashed] (X8)--(X9);
\draw[dashed] (X9)--(X10);
\draw[dashed] (X3)--(X8);
\draw[dashed] (X10)--(X7);
}

{
\draw (X3)--(X4);
\draw (X4)--(X5);
\draw (X5)--(X6);
\draw (X6)--(X7);
}

\end{tikzpicture}
\caption{\label{cycles}{\it Two copies of $C_8$ which overlap in 4 edges. Solid lines indicate overlaping edges ($E_0$), dashed and dotted lines indicate non-overlaping edges of the individual copies ($E_1,E_2$). Labeling every edge with 1 is a valid labeling for $F_4$. However, since $F'$ contains an odd cycle, this labeling is not valid for $F'$.}}
\end{center}
\end{figure}

Note that for a labeling of edges to be valid the number of ones assigned to each of these sets must have the same parity. 
Therefore, the previous calculation for $B_{F_\ell}$ already implies $B_{F_\ell}\geq B_{F'}$, as we may sum over at most as many terms as before (see Figure \ref{cycles}).

Now, by Lemma~\ref{2ndMM} it is sufficient to show that $B_{F_{\ell}}<\left(B_{C_k}\right)^2$ for $0<\ell<k-1$. 
Observe that by Equation \eqref{baseFl2} we have that $B_{F_{\ell}}\geq B_{F_{\ell-1}}$ if and only if
\begin{align}
 1+\left(\frac{\alpha-\beta}{\alpha+\beta}\right)^{2k-2\ell}&+2\left(\frac{\alpha-\beta}{\alpha+\beta}\right)^k\\
\nonumber &\geq (\alpha+\beta)\left(1+\left(\frac{\alpha-\beta}{\alpha+\beta}\right)^{2k-2\ell+2}+2\left(\frac{\alpha-\beta}{\alpha+\beta}\right)^k\right)
\end{align}
or equivalently
\begin{align}
\left(1-\frac{(\alpha-\beta)^2}{\alpha+\beta}\right)\left(\frac{\alpha-\beta}{\alpha+\beta}\right)^{2k-2\ell}&\geq (\alpha+\beta-1)\left(1+2\left(\frac{\alpha-\beta}{\alpha+\beta}\right)^k\right).\label{cycleEst}
\end{align}
Note that the left hand side of Inequality \eqref{cycleEst} increases as $\ell$ increases, while the right hand side does not depend on $\ell$. Thus it is enough to show that $B_{F_1}<\left(B_{C_k}\right)^2$ and $B_{F_{k-2}}<\left(B_{C_k}\right)^2$. We will distinguish three cases for this.
\begin{enumerate}
\item If $\alpha+\beta\leq 1$, then we obtain
\begin{align*}
B_{F_\ell}&\stackrel{\eqref{baseFl2}}{=}\frac{1}{2}(\alpha+\beta)^{2k-\ell}\left(1+\left(\frac{\alpha-\beta}{\alpha+\beta}\right)^{2k-2\ell}+2\left(\frac{\alpha-\beta}{\alpha+\beta}\right)^k\right)\\
&\leq \frac{1}{2}(\alpha+\beta)^{k}\left(1+\left(\frac{\alpha-\beta}{\alpha+\beta}\right)^{2k-2\ell}+2\left(\frac{\alpha-\beta}{\alpha+\beta}\right)^k\right)\\
&< (\alpha+\beta)^{k}\left(1+\left(\frac{\alpha-\beta}{\alpha+\beta}\right)^k\right)=(\alpha+\beta)^k+(\alpha-\beta)^k=B_{C_k},
\end{align*}
where the last inequality follows from $|\alpha-\beta|<|\alpha+\beta|$.
Since $B_{C_k}>1$ the claim follows. 

\item If $\alpha+\beta>1$ and $\alpha\geq\beta$ or $\alpha+\beta>1$ and $k$ is even, then we have
\begin{align*}
B_{F_\ell}&\stackrel{\eqref{baseFl2}}{=}\frac{1}{2}(\alpha+\beta)^{2k-\ell}\left(1+\left(\frac{\alpha-\beta}{\alpha+\beta}\right)^{2k-2\ell}+2\left(\frac{\alpha-\beta}{\alpha+\beta}\right)^k\right)\\
&< \frac{1}{2}(\alpha+\beta)^{2k}\left(1+1+2\left(\frac{\alpha-\beta}{\alpha+\beta}\right)^{2k}+4\left(\frac{\alpha-\beta}{\alpha+\beta}\right)^k\right)\\
&= \left((\alpha+\beta)^k+(\alpha-\beta^k)\right)^2=\left(B_{C_k}\right)^2.
\end{align*}

\item If $\alpha+\beta>1$, $\beta>\alpha$ and $k$ is odd. 
Note that under these conditions $(\alpha-\beta)^k<0$ and therefore \begin{equation}
\label{est}(\alpha+\beta)^2(\alpha-\beta)^k\leq \frac{(\alpha-\beta)^{k+4}}{(\alpha+\beta)^2}\, ,
\end{equation} implying
\begin{align*}
B_{F_{k-2}}&\stackrel{\eqref{baseFl1}}{=}\frac{1}{2}\left((\alpha+\beta)^{k+2}+(\alpha+\beta)^{k-2}( \alpha-\beta)^{4}+2(\alpha+\beta)^{2}(\alpha-\beta)^k\right)\\
&\stackrel{\eqref{est}}{\leq}\frac{1}{2}\left((\alpha+\beta)^k+(\alpha-\beta)^k\right)\left((\alpha+\beta)^2+\frac{(\alpha-\beta)^{4}} {(\alpha+\beta)^2}\right).
\end{align*}
Furthermore note that for every odd $i\in \mathbb{N}$ we have  $$(\alpha+\beta)^i+(\alpha-\beta)^i<(\alpha+\beta)^{i+2}+(\alpha-\beta)^{i+2},$$ 
and therefore $B_{F_{k-2}}<\left(B_{C_k}\right)^2$ follows if we can show 
\[\frac{1}{2}\left((\alpha+\beta)^2+\frac{(\alpha-\beta)^{4}} {(\alpha+\beta)^2}\right)<(\alpha+\beta)^3+(\alpha-\beta)^3.\]
On the one hand, if $\alpha\geq 1/5$ we have 
\begin{align*}
\frac{1}{2}((\alpha+\beta)^4+(\alpha-\beta)^4)&=\alpha^4+6\alpha^2\beta^2+\beta^4\\
&<(2\alpha+5\beta)\alpha^4+10(\alpha+\beta)\alpha^2\beta^2+5\alpha\beta^4\\
&\leq (\alpha+\beta)^5+\alpha^5-\beta^5\\
&\leq (\alpha+\beta)^5+(\alpha+\beta)^2(\alpha-\beta)^3.
\end{align*}
On the other hand, if $\alpha<1/3$, then we have
\begin{align*}
\frac{1}{2}\left((\alpha+\beta)^2+\frac{(\alpha-\beta)^{4}} {(\alpha+\beta)^2}\right)&=\frac{\alpha^4+6\alpha^2\beta^2+\beta^4}{\alpha^2+2\alpha\beta+\beta^2}<1\, ,
\end{align*}
since $\alpha^4<\alpha^2$, $\beta^4<\beta^2$ and $6\alpha^2\beta^2<2\alpha\beta$.

\hspace{16pt} Hence, all that is left is to show that $B_{F_1}<\left(B_{C_k}\right)^2$. Assume $(\alpha+\beta)^{k-1}\leq 2$. Then we have
\begin{align*}
B_{F_1}&=\frac{1}{2}(\alpha+\beta)^{k-1}((\alpha+\beta)^k+(\alpha-\beta)^k)\\
&\hspace{16pt}+\frac{1}{2}(\alpha+\beta)(\alpha-\beta)^k((\alpha+\beta)^{k-2}+(\alpha-\beta)^{k-2})\\
&\leq\frac{1}{2}(\alpha+\beta)^{k-1}((\alpha+\beta)^k+(\alpha-\beta)^k)\\
&\leq(\alpha+\beta)^k+(\alpha-\beta)^k=B_{C_k}<\left(B_{C_k}\right)^2.
\end{align*}
However, if $(\alpha+\beta)^{k-1}>2$, then we have  
\begin{align*}
\left(\frac{\alpha-\beta}{\alpha+\beta}\right)^{2k-2}-2\left(\frac{\alpha-\beta}{\alpha+\beta}\right)^k&= \frac{(\alpha-\beta)^{2k-2}-2(\alpha+\beta)^{k-2}(\alpha-\beta)^{k}}{(\alpha+\beta)^{2k-2}}\\ 
&\leq \frac{1-2(\alpha^2-\beta^2)^{k-2}(\alpha-\beta)^{2}}{(\alpha+\beta)^{2k-2}}\\ 
&\leq \frac{1+2(\alpha-\beta)^{2}}{(\alpha+\beta)^{2k-2}}\leq \frac{3}{4}<1.
\end{align*}
Thus we have, as desired,
\begin{align*}
B_{F_1}&=\frac{1}{2}(\alpha+\beta)^{2k-1}\left(1+\left(\frac{\alpha-\beta}{\alpha+\beta}\right)^{2k-2}+2\left(\frac{\alpha-\beta}{\alpha+\beta}\right)^k\right)\\
&<\frac{1}{2}(\alpha+\beta)^{2k-1}\left(2+4\left(\frac{\alpha-\beta}{\alpha+\beta}\right)^k\right)\\
&<\frac{1}{2}(\alpha+\beta)^{2k}\left(2+2\left(\frac{\alpha-\beta}{\alpha+\beta}\right)^{2k}+4\left(\frac{\alpha-\beta}{\alpha+\beta}\right)^k\right)=\left(B_{C_k}\right)^2.\, \qedhere
\end{align*}
\end{enumerate}
\end{proof}

\section{Structural properties: proof of Theorem~\ref{ndistance}}\label{structure}

In this section we study the stochastic Kronecker graph when $\alpha=\gamma$ and $\alpha+\beta>1$. In particular, we are interested in the neighborhood $N(u)$ of a vertex $u$ and its degree $d(u)=\left|N(u)\right|$. We prove Theorem~\ref{ndistance} stating that all vertices have the same degree and for almost all edges the end points have asymptotically the same Hamming distance from one another.

\begin{proof}[Proof of Theorem~\ref{ndistance}]

Recall that for every vertex $u$ we have 
$$\mathbb{E}(d(u))=(1+o(1))(\alpha+\beta)^n$$
and 
$$\mathrm{Var}\left(d\left(v\right)\right)=(1+o(1))\mathbb{E}\left(d\left(v\right)\right),$$
by Lemma~\ref{expectation}. Thus, for a fixed vertex $u$ the Chernoff Bound implies that 
\[\mathbb{P}\left(\left|d(u)-\mathbb{E}(d(u))\right|>\log{n}\sqrt{n\mathbb{E}(d(u))}\right)=o(\exp(-n)).\]
The first statement follows by applying the union bound.

In order to prove the second statement we now define for each $k\in\{0,\dots,n\}$ a random variable $Y_{u,k}$ that counts the number of neighbors of $u$ at Hamming distance $k$ and recall that by representing $Y_{u,k}$ as a sum of indicator variables we get \begin{equation}\label{HamExp}
\mathbb{E}\left(Y_{u,k}\right)=\binom{n}{k}\alpha^{n-k}\beta^k=(\alpha+\beta)^n\, \mathbb{P}\left(\mathrm{Bin}\left(n,\frac{\beta}{\alpha+\beta}\right)=k\right).
\end{equation} Now we define a subset of the natural numbers
\begin{equation}\label{bad}
\mathcal{J}=\left\{k\in\{0,\dots,n\}:\left|k-\frac{\beta}{\alpha+\beta}n\right|>\sqrt{2\frac{\beta}{\alpha+\beta}}\log{n}\sqrt{n}\right\},
\end{equation} and call a neighbor $w$ of $u$ \emph{bad} if $$H(w,u)\in\mathcal{J} ,$$ i.e.\ $w$ is either too far or too close in Hamming distance to $u.$ Then we set
\[Y_u=\sum_{k\in\mathcal{J}}Y_{u,k},\] i.e.\ $Y_u$ is the random variable that counts the number of bad neighbors of $u.$ By linearity of expectation we get 
\begin{align*}
\mathbb{E}\left(Y_u\right)&=\sum_{k\in\mathcal{J}}\mathbb{E}\left(Y_{u,k}\right)\stackrel{\eqref{HamExp}}{=}\left(\alpha+\beta\right)^n\, \mathbb{P}\left(\mathrm{Bin}\left(n,\frac{\beta}{\alpha+\beta}\right)\in\mathcal{J}\right).
\end{align*} 
Moreover, by Definition \eqref{bad}, the Chernoff Bound provides the following upper bound for this probability
\begin{align*}
\mathbb{P}\left(\left|\mathrm{Bin}\left(n,\frac{\beta}{\alpha+\beta}\right)-\frac{\beta}{\alpha+\beta}n\right|>\sqrt{2\frac{\beta}{\alpha+\beta}}\log{n}\sqrt{n}\right)\leq 2e^{-(1+o(1))\log^2{n}},
\end{align*}
and thus we get for the expected number of bad neighbors of $u$
\begin{align*}
\mathbb{E}(Y_u)&\leq 2\exp\left(-(1+o(1))\log^2{n}\right)(\alpha+\beta)^n.
\end{align*}
Furthermore, note that the random variable $Y_u$ is the sum of independent Bernoulli random variables, hence $\mathrm{Var}\left(Y_u\right)\leq \mathbb{E}\left(Y_u\right),$ and thus the Chernoff Bound implies 
\begin{align*}
\mathbb{P}\left(Y_u>\frac{(\alpha+\beta)^n}{n}\right)&< \exp\left(-\frac{(1+o(1))(\alpha+\beta)^{2n}\exp((1+o(1))\log^2{n})}{3n^2(\alpha+\beta)^n}\right)\\
&=o(\exp(-n)).
\end{align*} Hence, with probability at least $1-o(\exp(-n))$, almost all neighbors of $u$ are not bad, i.e.\ their Hamming distance is not \enquote{too far} from $\frac{\beta}{\alpha+\beta}n$. Applying the union bound completes the proof.
\end{proof}

We have just proved that for every vertex $u$ in the stochastic Kronecker graph $K(n,P),$ with $\alpha=\gamma$ and $\alpha+\beta>1,$ that almost all of its neighbors have asymptotically the same Hamming distance from $u.$ Intuitively, for \enquote{small enough} $\alpha>0$ a.a.s.\ there should not even be a single edge between two vertices with a \enquote{small} Hamming distance, since the probability of having a fixed edge of that sort decreases rapidly with $\alpha\to 0$. More precisely, we will now show that if $\alpha<1/2$ then a.a.s.\ there are no edges between vertices that have Hamming distance at most $cn$ for some $0<c<\beta/(\alpha+\beta)<1.$ Similarly if $\beta<1/2$ then a.a.s.\ there are no edges between vertices that have Hamming distance at least $cn$ for some $0<\beta/(\alpha+\beta)<c<1.$ In fact these constants $c$ can be determined as the solutions of the following equation.  

\begin{lemma}\label{solutions}
Assume that $\alpha+\beta>1$. Then there is at most one solution $c\in(0,1)$ of the equation
\[\left(\frac{\beta}{c}\right)^c\left(\frac{\alpha}{1-c}\right)^{1-c}= \frac{1}{2}.\]
In addition, if $\alpha< 1/2$, then there is exactly one solution and $$0< c < \beta/(\alpha+\beta)<1.$$ However, if $\beta< 1/2$, then there is exactly one solution and $$0<\beta/(\alpha+\beta)<c< 1.$$
\end{lemma}

\begin{proof}
Let us define a function $\psi:(0,1)\to\mathbb{R}$ by setting
\[\psi(c)=\left(\frac{\beta}{c}\right)^c\left(\frac{\alpha}{1-c}\right)^{1-c}\]
and note that $\psi$ is continuous on $(0,1)$.
Furthermore, $\psi$ is differentiable on $(0,1)$ and we have that
\begin{align*}
\psi'(c)&=\ln{\left(\frac{\beta(1-c)}{\alpha c}\right)}\left(\frac{\beta}{c}\right)^c\left(\frac{\alpha}{1-c}\right)^{1-c}.
\end{align*}
Note that $\psi'(c)<0$ if and only if $\beta(1-c)<\alpha c$, or equivalently $c>\beta/(\alpha+\beta),$ and $\psi'(c)>0$ if and only if $\beta(1-c)>\alpha c$ or equivalently $c<\beta/(\alpha+\beta)$. Therefore $\psi(c)$ strictly increases on $(0,\beta/(\alpha+\beta))$ and strictly decreases  on $(\beta/(\alpha+\beta),1)$. Furthermore $\psi(c)\rightarrow \alpha$ as $c\rightarrow 0$, $\psi(c)\rightarrow \beta$ as $c\rightarrow 1$. Let $c_0=\beta/(\alpha+\beta),$ i.e.\ $c_0$ is the point where $\psi$ attains its unique, global maximum and note that we have
\[\psi\left(c_0\right)=\left(\frac{\beta}{c_0}\right)^{c_0}\left(\frac{\alpha}{1-c_0}\right)^{1-c_0}= (\alpha+\beta)^{c_0}(\alpha+\beta)^{1-c_0}=(\alpha+\beta)>1.\]
Therefore, if $\alpha< 1/2$, then there is a solution to $\psi(c)=1/2$ on $(0,\beta/(\alpha+\beta))$ and if $\beta< 1/2$ there is a solution on $(\alpha/(\alpha+\beta),1)$. The statement follows from the fact that at most one of $\alpha,\beta$ can be $< 1/2$.
\end{proof}

From Theorem~\ref{ndistance} and by application of Lemma~\ref{binomappr} we deduce the following theorem.

\begin{theorem}\label{noCloseNeighbor}
Assume that $\alpha=\gamma$ and $\alpha+\beta>1$. Let $c\in(0,1)$ satisfy
\[\left(\frac{\beta}{c}\right)^c\left(\frac{\alpha}{1-c}\right)^{1-c}= \frac{1}{2}.\]
Then, if $\alpha< 1/2$, a.a.s.\ no vertex $u$ has a neighbor $w$ such that $$H(u,w)<cn.$$ However, if $\beta< 1/2$, then a.a.s.\ no vertex $u$ has a neighbor $w$ such that $$H(u,w)>cn.$$
\end{theorem}

\begin{proof}
Assume that $\alpha< 1/2$. The expected number of neighbors of vertex $u$ at distance at most $cn$ is
\[\sum_{i=0}^{cn}\binom{n}{i}\alpha^{n-i}\beta^i.\]
Lemma \ref{solutions} implies that $c<\beta/(\alpha+\beta)$. Therefore, according to Lemma \ref{binomappr} we have 
\begin{align*}
\sum_{i=0}^{cn}\binom{n}{i}\alpha^{n-i}\beta^i &\stackrel{L. \ref{binomappr}}{=}O\left(\binom{n}{cn}\alpha^{(1-c)n}\beta^{cn}
\right)\\
&=o\left(\left(\frac{\alpha^{1-c}\beta^c}{c^c(1-c)^{1-c}}\right)^{n}\right)\\
&\stackrel{L. \ref{solutions}}{=}o(2^{-n}).
\end{align*}
The result follows by applying the union bound to the $2^n$ vertices and Markov's inequality. The case when $\beta< 1/2$ is analogous.
\end{proof}

Furthermore, the constant $c$ in Theorem~\ref{noCloseNeighbor} is optimal in the following sense.

\begin{theorem}
Assume that $\alpha=\gamma$ and $\alpha+\beta>1$. Let $c\in(0,1)$ satisfy
\[\left(\frac{\beta}{c}\right)^c\left(\frac{\alpha}{1-c}\right)^{1-c}= \frac{1}{2}.\]
Then, if $\alpha< 1/2$, a.a.s.\ there is an edge connecting two vertices $u,v$ such that $$H(u,w)=cn+\log^2{n},$$
and if $\beta< 1/2$, then a.a.s.\ there is an edge connecting two vertices $u,v$ such that $$H(u,w)=cn-\log^2{n}.$$
\end{theorem}

\begin{proof}
Assume $\alpha< 1/2$ and note that this implies that $\alpha<\beta$. The probability $q$ that no two vertices at distance $cn+\log^2{n}$ are connected satisfies
\begin{align*}
q&=\left(1 \vphantom{-\alpha^{(1-c)n}\beta^{cn}\left(\frac{\beta}{\alpha}\right)^{\log^2{n}}}\right. \left. -\alpha^{(1-c)n}\beta^{cn}\left(\frac{\beta}{\alpha}\right)^{\log^2{n}}\right)^{2^{n-1}\binom{n}{cn+\log^2{n}}}\\
&\leq \exp\left(- \alpha^{(1-c)n}\beta^{cn}\left(\frac{\beta}{\alpha}\right)^{\log^2{n}}2^{n-1}(1+o(1))\binom{n}{cn}\right)\\
&\leq \exp\left(-\Theta\left(\frac{1}{\sqrt{n}}\right)\left(2\frac{\alpha^{(1-c)}\beta^{c}}{c^c(1-c)^{1-c}}\right)^n \left(\frac{\beta}{\alpha}\right)^{\log^2{n}}\right)=o(1),
\end{align*}
as $\beta>1/2> \alpha$. Hence, a.a.s.\ there is an edge between two vertices at distance $cn+\log^2n.$ The case $\beta< 1/2$ is analogous.
\end{proof}

\bibliographystyle{plain}
\bibliography{references}

\end{document}